\documentclass[hidelinks,onefignum,onetabnum,oneside]{siamart251216}

\usepackage{mathtools}
\usepackage{amsmath,amssymb,amsfonts}
\usepackage{bm}

\usepackage{hyperref}
\usepackage{cleveref}
\usepackage{cancel}
\usepackage{array}
\usepackage{subcaption}
\usepackage{standalone}
\usepackage{tikz}
\usetikzlibrary{spy,shapes,arrows,positioning,calc,decorations.markings}
\usepackage{pgfplots}
\pgfplotsset{compat=1.18}
\usepgfplotslibrary{colorbrewer}
\usepackage{pgfplotstable}
\usepackage{booktabs}
\usepackage{colortbl}
\usepackage{makecell}
\usepackage{enumitem}
\newlist{identities}{enumerate}{1}
\setlist[identities,1]{
  label={\textbf{ID \arabic*}},
  ref=\arabic*, 
  wide,itemsep=0pt,topsep=0pt}
\creflabelformat{identitiesi}{#2\textup{#1}#3} 
\crefname{identitiesi}{ID}{IDs} 
\Crefname{identitiesi}{ID}{IDs}

\usepackage{lipsum}
\usepackage{graphicx}
\usepackage{epstopdf}
\usepackage{amsopn}
\usepackage{matlab-prettifier}

\usepackage{algorithm,algpseudocode}

\newcommand*\circled[1]{\tikz[baseline=(char.base)]{\node(char)[shape=rounded rectangle,draw,inner sep=0.6pt,minimum height=1.5ex]{#1};}} %
\newcommand\kronF[2]{{#1}^{\circled{\tiny{\ensuremath{#2}}}}} 
\renewcommand\Vec[1][]{\textnormal{vec}\ensuremath{\if$#1$ \else \left[#1\right]\fi}}

\newcommand{\real}{\mathbb{R}}

\newcommand{\rd}{\text{\upshape d}} 

\newcommand{\bzero}{\ensuremath{\mathbf{0}}} 

\newcommand{\bA}{\ensuremath{\mathbf{A}}}
\newcommand{\bB}{\ensuremath{\mathbf{B}}}
\newcommand{\bC}{\ensuremath{\mathbf{C}}}
\newcommand{\bD}{\ensuremath{\mathbf{D}}}

\newcommand{\bF}{\ensuremath{\mathbf{F}}}
\newcommand{\bG}{\ensuremath{\mathbf{G}}}
\newcommand{\bH}{\ensuremath{\mathbf{H}}}
\newcommand{\bI}{\ensuremath{\mathbf{I}}}
\newcommand{\bJ}{\ensuremath{\mathbf{J}}}
\newcommand{\bK}{\ensuremath{\mathbf{K}}}

\newcommand{\bM}{\ensuremath{\mathbf{M}}}

\newcommand{\bP}{\ensuremath{\mathbf{P}}}

\newcommand{\bT}{\ensuremath{\mathbf{T}}}

\renewcommand{\bf}{\ensuremath{\mathbf{f}}}
\newcommand{\bg}{\ensuremath{\mathbf{g}}}
\newcommand{\bh}{\ensuremath{\mathbf{h}}}

\newcommand{\bp}{\ensuremath{\mathbf{p}}}
\newcommand{\bq}{\ensuremath{\mathbf{q}}}

\newcommand{\bu}{\ensuremath{\mathbf{u}}}
\newcommand{\bv}{\ensuremath{\mathbf{v}}}
\newcommand{\bw}{\ensuremath{\mathbf{w}}}
\newcommand{\bx}{\ensuremath{\mathbf{x}}}
\newcommand{\by}{\ensuremath{\mathbf{y}}}
\newcommand{\bz}{\ensuremath{\mathbf{z}}}

\newcommand{\bPsi}{\mbox{\boldmath $\Psi$}}

\newcommand{\bPhi}{\mbox{\boldmath $\Phi$}}
\newcommand{\bvarphi}{\mbox{\boldmath $\varphi$}}

\newcommand{\bsigma}{\ensuremath{\boldsymbol{\sigma}}}

\newcommand{\cA}{\ensuremath{\mathcal{A}}}

\newcommand{\cE}{\ensuremath{\mathcal{E}}}

\newcommand{\cL}{\ensuremath{\mathcal{L}}}
\newcommand{\cM}{\ensuremath{\mathcal{M}}}

\newcommand{\cP}{\ensuremath{\mathcal{P}}}

\newcommand{\cT}{\ensuremath{\mathcal{T}}}

\newcommand{\cW}{\ensuremath{\mathcal{W}}}

\def\cool#1{\textcolor{blue}{#1}}

\usepackage{changepage}
\usepackage{mparhack}
\strictpagecheck 

\newsiamremark{remark}{Remark}
\newsiamremark{hypothesis}{Hypothesis}
\crefname{hypothesis}{Hypothesis}{Hypotheses}
\newsiamthm{claim}{Claim}
\newsiamremark{fact}{Fact}
\crefname{fact}{Fact}{Facts}

\ifpdf
  \DeclareGraphicsExtensions{.eps,.pdf,.png,.jpg}
\else
  \DeclareGraphicsExtensions{.eps}
\fi

\headers{Nonlinear Balanced Truncation}{N. A. Corbin and B. Kramer}

\title{Nonlinear balanced truncation model reduction through scalable Taylor series\thanks{Submitted to the editors April 23rd, 2026.
    \funding{This work was funded by AFOSR award no.~FA9550-24-1-0237.}}}

\author{Nicholas A. Corbin \and Boris Kramer\thanks{Department of Mechanical and Aerospace Engineering, University of California San Diego, La Jolla, CA
    (\email{ncorbin@ucsd.edu}, \email{bmkramer@ucsd.edu}, \url{https://kramer.ucsd.edu/}).}}

\graphicspath{{./figures/}}

\renewenvironment{quotation}
  {\list{}{\listparindent 0em%
    \itemindent    \listparindent
    \leftmargin    .5cm  
    \rightmargin   .5cm} 
    \item\relax}
  {\endlist}
\newcounter{algsubstate}
\renewcommand{\thealgsubstate}{\alph{algsubstate}}
\newenvironment{algsubstates}
  {\setcounter{algsubstate}{0}%
   \renewcommand{\State}{%
     \stepcounter{algsubstate}%
     \Statex {\footnotesize\thealgsubstate:}\space}}
  {}

\ifpdf
\hypersetup{
  pdftitle={Computation of Reduced-Order Models via Nonlinear Balanced Truncation},
  pdfauthor={N. A. Corbin and B. Kramer}
}
\fi

\begin{document}

\maketitle

\begin{abstract}
  The theory of nonlinear balanced truncation provides a system-theoretic framework for model reduction that preserves important properties such as stability, controllability, and observability.
  We present a scalable algorithm for computing reduced-order models
  based on the nonlinear balancing theory.
  The approach is based on polynomial approximations
  using the Kronecker product representation, building on recent numerical linear algebra advancements to enable scalability.
  We derive polynomial approximations for the balancing transformation and the explicit balanced realization of the full-order model, which yields true nonlinear reduced-order models upon truncation of redundant state components.
  The proposed tools are tested on various examples, demonstrating a nuanced perspective of the benefits and limitations of nonlinear balancing not shown in the existing literature.
\end{abstract}



\section{Introduction}
With the advent of modern control theory in the 1960s by the pioneering works of Kalman \cite{Kalman1960,Kalman1960a,Kalman1960b,Kalman1960c,Kalman1960d,Kalman1961}, who recognized the importance of the preceding works of Lyapunov \cite{Lyapunov1892} nearly a hundred years prior,
the treatment of control systems underwent a shift from analysis and design in the frequency domain, characterized by the transfer function, to analysis and design in the time domain, characterized by the state-space representation.
Despite the strengths of classical techniques, modern techniques enable a more straightforward treatment of certain classes of systems, most famously multi-input/multi-output systems and nonlinear systems.
There is, however, a peculiarity regarding the time-domain approach: while the external description via the transfer function is unique, the internal description via a state-space representation is not.
Not only that: it is possible to have state-space realizations of differing orders that nonetheless share the same transfer function.
In subsequent papers \cite{Kalman1963a,Kalman1965}, Kalman addressed this issue with the topic of minimal realization theory, through which he introduced many of the foundational concepts still used to describe control systems, such as controllability and observability.
The introduction of minimal realizations, or ``irreducible realizations'' as Kalman originally called them, planted the seed for what would eventually form the field of model order reduction.
Almost twenty years later, in 1981, Moore formalized the theory of balanced truncation for linear time-invariant systems \cite{Moore1981}.
The opening of Moore's paper is particularly salient:
\begin{quotation}
  ``\textit{A common and quite legitimate complaint directed toward multivariable control literature is that the apparent strength of the theory is not accompanied by strong numerical tools. Kalman's minimal realization theory, for example, offers a beautifully clear picture of the structure of linear systems. Practically every linear systems text gives a discussion of controllability, observability, and minimal realization. The associated textbook algorithms are far from satisfactory, however, serving mainly to illustrate the theory with textbook examples.}''
\end{quotation}
Moore elaborated that the unsatisfactory nature of minimal realization theory and algorithms at the time stemmed from a rigid, binary perspective in which
model order was identified
as the number of non-zero singular values of the Hankel matrix.
Moore recognized that,
especially in terms of implementation on digital computers with finite precision,
it is not just \textit{zero} singular values, but also \textit{near zero} singular values, that indicate reducibility.
Moore's procedure for constructing a minimal realization by ``balancing'' the controllability and observability energies and then ``truncating'' the Hankel singular values below a specified tolerance, rather than just those that are zero,
became known as balanced truncation.

Following Moore's work, many other contributions were made to generalize and refine the balancing theory and algorithms for linear systems.
Notable contributions include the square-root balancing algorithm \cite{Laub1987,Tombs1987}, which enabled removing Moore's assumption that the full-order model already be minimal, and closed-loop balancing formulations \cite{Jonckheere1983,Mustafa1991}, which addressed unstable systems.
In 1993, Scherpen extended the balancing theory to nonlinear control-affine systems \cite{Scherpen1993,Scherpen1994a}.
Since then, the \textit{theory} of balancing for nonlinear systems has been studied thoroughly, with subsequent publications uncovering connections between the Hankel operator \cite{Gray1998}, minimality \cite{Gray2000}, different definitions of singular value functions \cite{Gray2001}, and their implications on nonlinear systems theory.
The issue of the development of \textit{numerical tools} to implement nonlinear balancing also arose at this time;
however, unlike the linear theory, which requires
the solutions of algebraic matrix equations for which solvers are readily available,
the nonlinear theory involves the solution of
notoriously challenging
nonlinear Hamilton-Jacobi-Bellman partial differential equations (HJB PDEs)\footnote{The complexity of solving HJB PDEs has even stimulated the development of various alternative balancing theories \cite{Condon2005,Gray2006,Besselink2014,Kawano2017,Benner2024} and of course many data-driven approaches \cite{Lall2002,Bouvrie2017} in order to entirely avoid solving these equations.}.
Efforts have been made towards the development of computational methods for nonlinear balancing, including discretized Monte Carlo-based optimization techniques \cite{Newman1998,Newman1999} and Taylor series-based methods \cite{Fujimoto2008a,Krener2008}.
The latter in particular showed promise on low-dimensional examples, the highest-order example shown being the triple pendulum of state dimension $n=6$ in \cite{Krener2008}.
However, the lack of an open-source, scalable, general purpose implementation of nonlinear balanced truncation for reproducibility and application to other higher-dimensional systems has left much to be desired.

Three main challenges must be overcome to derive a nonlinear balanced truncation model.
First, solutions must be computed to the HJB PDEs for the controllability and observability energy functions.
Second, a nonlinear coordinate transformation must be computed that puts the system in input-normal/output-diagonal form.
Third, an additional scaling transformation must be computed, after which the balanced realization and
the construction of the reduced-order model (ROM) must be numerically implemented.
The primary contribution of this work is a thorough treatment of this third challenge based on the Kronecker product representation of polynomials which provides the first scalable computation of its kind, building on our previous works \cite{Corbin2025a,Corbin2025b} which address the first two challenges.
The result is a scalable, general-purpose computational approach for nonlinear balancing applicable to a variety of problems.
Our open-source $\textsc{Matlab}$ software package is available at\footnote{\cool{The code will be archived in a Zenodo repository upon final submission}}
\begin{center}
  \texttt{\url{https://github.com/cnick1/NLbalancing}}
\end{center}

The paper is organized as follows.
We summarize the relevant portions of the nonlinear balancing theory in \cref{sec:nlbt-theory}.
Our main results are in \cref{sec:polynomial-balancing},
in which we derive the explicit formulae for polynomial approximations to the balancing transformation, the dynamics in the balanced realization, and an explicit reduced-order model.
We present several numerical examples in \cref{sec:examples}
to provide a clear perspective of the performance of the method, and the conclusions follow in
\cref{sec:conclusions}.

\section{Nonlinear balancing theory and Kronecker product preliminaries}\label{sec:nlbt-theory}
Consider the nonlinear control-affine dynamical system
\begin{equation}\label{eq:FOM-NL}
  \begin{split}
    \dot{\bx}(t) & = \bf(\bx(t))  + \bg(\bx(t)) \bu(t), \qquad \bx(0) = \bx_0 \in \real^n, \\
    \by(t)       & = \bh(\bx(t)),
  \end{split}
\end{equation}
where $\bx(t) \in \real^n$ is the state
and
$\bf \colon \real^n \mapsto \real^n$, $\bg \colon \real^n \mapsto \real^{n \times m}$, and $\bh \colon \real^n \mapsto \real^{p}$
represent the nonlinear drift, input, and output mappings.
The inputs to the system are $\bu(t) \in \real^m$, and the outputs are given by $\by(t) \in \real^p$.
We henceforth refer to the dynamical system \cref{eq:FOM-NL} as the full-order model (FOM).

We wish to derive a ROM approximation to \cref{eq:FOM-NL} as
\begin{equation}\label{eq:ROM-NL}
  \begin{split}
    \dot{\bx}_r(t) & = \bf_r(\bx_r(t))  + \bg_r(\bx_r(t)) \bu(t), \qquad \bx_r(0) = {\bx_r}_0  \in \real^r, \\
    \by(t)         & = \bh_r(\bx_r(t)),
  \end{split}
\end{equation}
where $\bx_r(t) \in \real^r$, $\bf_r:\real^r \mapsto \real^r$, $\bg_r:\real^r \mapsto \real^{r\times m}$, and $\bh_r:\real^r \mapsto \real^p$,
that preserves properties such as stability, controllability, observability \cite[Thm. 10]{Fujimoto2010}, and the input/output behavior of the FOM.
Scherpen's nonlinear balancing theory \cite{Scherpen1993,Fujimoto2010} achieves this via a nonlinear coordinate transformation, described in \cref{sec:balancing-transformation}, which puts the system in a balanced realization so that the reduced-order dynamics can be obtained via truncation of certain state components, as described in \cref{sec:balanced-realization}.
Next, we provide an abridged overview of the relevant theory to construct this truncated approximation; the interested reader should look to \cite{Fujimoto2010} and the previous works cited therein for a more thorough discussion of the theory.

\subsection{Observability and controllability energy functions}
The controllability and observability energy functions for an asymptotically stable nonlinear system \cref{eq:FOM-NL} are defined in \cite{Scherpen1993} as
\begin{align*}
  \cE_c(\bx_0) & \coloneqq
  \!\!\!\! \min_{\substack{\bu \in L_{2}
      (-\infty, 0)
  \\ \bx(-\infty) = \bzero,\,\,   \bx(0) = \bx_0}}
  \!\!
  \frac{1}{2} \int\displaylimits_{-\infty}^{0} \Vert \bu(t) \Vert^2 \rd t,
  \hspace{-2cm}
  \\
  \cE_o(\bx_0) & \coloneqq \frac{1}{2} \int\displaylimits_{0}^{\infty} \Vert \by(t) \Vert^2  \rd t,
  \begin{split}
     & \qquad\bx(0) = \bx_0, \\&\bu(t) \equiv \bzero, \,\,0 \leq t < \infty.
  \end{split}
\end{align*}
We assume that these functions exist, are finite\footnote{These assumptions are equivalent to assuming controllability and observability.},
and are smooth in a neighborhood of the origin.
Under the assumption that $0$ is an asymptotically stable equilibrium of $-\left(\bf(\bx)+\bg(\bx)\bg(\bx)^\top\frac{\partial^\top\cE_c(\bx)}{\partial \bx}\right)$,
the controllability energy
$\cE_c(\bx)$
is the unique solution to the
Hamilton-Jacobi equation
\begin{equation}\label{eq:HJ-Equation}
  \begin{split}
     & 0 =  \frac{\partial \cE_c(\bx)}{\partial \bx} \bf(\bx) + \frac{1}{2}  \frac{\partial \cE_c(\bx)}{\partial \bx} \bg(\bx) \bg(\bx)^\top \frac{\partial^\top \cE_c(\bx)}{\partial \bx},
    \qquad \cE_c(\bzero) = 0,
  \end{split}
\end{equation}
and the observability energy $\cE_o(\bx)$ is the unique solution to the nonlinear Lyapunov-like equation
\begin{equation}\label{eq:Nonlin-Lyap}
  \begin{split}
     & 0 =  \frac{\partial \cE_o(\bx)}{\partial \bx} \bf(\bx) + \frac{1}{2}\bh(\bx)^\top \bh(\bx),
    \qquad \cE_o(\bzero) = 0.
  \end{split}
\end{equation}
In \cite{Corbin2025a}, we developed solutions to these PDEs via series expansions using scalable tensor-based algorithms for models with state dimensions as high as $n=1080$.

\subsection{Balancing transformation}\label{sec:balancing-transformation}
In order to facilitate model order reduction,
the nonlinear balancing theory describes how to achieve a Kalman-like decomposition via a nonlinear coordinate transformation.
This enables truncating
the least observable/controllable states as determined by the magnitude of the singular value functions.
Similar to the linear balancing theory, the nonlinear balancing transformation is constructed as the composition of two transformations: the input-normal/output-diagonal transformation, and a subsequent scaling transformation.
The theory for these two transformations is described next.
\begin{theorem}\cite[Thm. 8]{Fujimoto2010} \label{thm:inputnormaloutputdiagonal}
  Suppose the Jacobian linearization of the nonlinear system \cref{eq:FOM-NL} is minimal, asymptotically stable, and has distinct Hankel singular values.
  Then there is a neighborhood $\cW$ of the origin and a smooth coordinate transformation $\bx = \bPhi(\bz)$ on $\cW$ with $\bz = [z_1, z_2, \dots, z_n]$ such that the controllability and observability energy functions have \emph{input-normal/output-diagonal} form:
  \begin{align}
    \cE_c(\bPhi(\bz)) & = \frac{1}{2} \sum_{i=1}^n  z_i^2,
                      &
    \cE_o(\bPhi(\bz)) & = \frac{1}{2} \sum_{i=1}^n  z_i^2 \sigma_i^2(z_i). \label{eq:in-od-definition}
  \end{align}
\end{theorem}
The input-normal/output-diagonal transformation shifts all of the observability and controllability information into the observability energy function in the transformed coordinates.
In this way, it transforms the observability and controllability information, which are coordinate dependent quantities, into the system invariant-like\footnote{For linear systems these are unique; for nonlinear systems, uniqueness breaks down \cite{Corbin2025b}.} quantities $\sigma_i(z_i)$ known as the singular value functions.
In \cite{Corbin2025b}, we expand on the approach proposed in \cite{Corbin2025a} with a procedure for computing the input-normal/output-diagonal transformation via polynomial approximation.
The input-normal/output-diagonal transformation is followed by a scaling transformation based on the singular value functions in order to put the system in the balanced form, as described next.
\begin{theorem}\cite[Thm. 9]{Fujimoto2010} \label{thm:balancing-transformation}
  Suppose the Jacobian linearization of the nonlinear system \cref{eq:FOM-NL} is minimal, asymptotically stable, and has distinct Hankel singular values.
  Then there is a neighborhood $\cW$ of the origin and a smooth coordinate transformation $\bx = \bar\bPhi(\bar\bz)$ on $\cW$ with $\bar\bz = [\bar z_1, \bar z_2, \dots, \bar z_n]$ such that the controllability and observability energy functions have \emph{balanced} form:
  \begin{align}
    \cE_c(\bar\bPhi(\bar\bz)) & = \frac{1}{2} \sum_{i=1}^n  \frac{\bar z_i^2}{\bar\sigma_i(\bar z_i)},
                              &
    \cE_o(\bar\bPhi(\bar\bz)) & = \frac{1}{2} \sum_{i=1}^n  \bar z_i^2 \bar\sigma_i(\bar z_i). \label{eq:balanced-definition}
  \end{align}
\end{theorem}
The full balancing transformation is the composition of the input-normal/output-diagonal and scaling transformations.
To reiterate, this coordinate transformation is important because it provides a partitioning for the dynamics that facilitates model order reduction by truncating the least observable/controllable states.
The next corollary gives the explicit form for the balancing transformation, which is presented in the proof of \cite[Thm. 9]{Fujimoto2010}.
\begin{corollary}\cite{Fujimoto2010}\label{cor:scaling-1}
  Consider the input-normal/output-diagonal transformation described in \cref{thm:inputnormaloutputdiagonal}, along with the scaling transformation $\bz = \bvarphi(\bar\bz)$ given by
  \begin{align*}
    \varphi_i(\bar \bz) = \frac{1}{\sqrt{\bar\sigma_i (\bar z_i)}} \bar z_i \quad (= z_i), \qquad\qquad
    \varphi_i^{-1}(\bz) = z_i \sqrt{\sigma_i (z_i)} \quad (= \bar z_i),
  \end{align*}
  where $\bar\sigma_i (\bar z_i) = \sigma_i ( z_i)$.
  The balancing transformation $\bar{\bPhi} = \bPhi \circ \bvarphi: \real^n \mapsto \real^n$ is given by the composition of the input-normal/output-diagonal transformation and the scaling transformation as
  $
    \bx = \bar\bPhi(\bar\bz) = \bPhi(\bvarphi(\bar\bz)).
  $
\end{corollary}

\subsection{Balanced realization and model order reduction by truncation}\label{sec:balanced-realization}
Inserting the coordinate transformation $\bx = \bar\bPhi(\bar\bz)$ into the FOM~\cref{eq:FOM-NL},
and assuming
that the Jacobian of the balancing transformation is invertible, the balanced realization can be written in control-affine form
in the transformed coordinates $\bar{\bz}$ as
\begin{equation}\label{eq:FOM-balanced}
  \begin{split}
    \dot{\bar\bz} & = \bar\bf(\bar\bz) + \bar\bg(\bar\bz) \bu(t), \qquad \bar\bz(0) = \bar\bz_0 = \bar\bPhi^{-1}(\bx_0) \in \real^n,
    \\
    \by           & = \bar\bh(\bar\bz),
  \end{split}
\end{equation}
where
\begin{align}\label{eq:balanced-realization-operators}
  \bar\bf(\bar\bz) & \coloneqq \left[\frac{\partial \bar\bPhi(\bar\bz)}{\partial \bar\bz}\right]^{-1}  \bf(\bar\bPhi(\bar\bz)),
                   &
  \bar\bg(\bar\bz) & \coloneqq \left[\frac{\partial \bar\bPhi(\bar\bz)}{\partial \bar\bz}\right]^{-1} \bg(\bar\bPhi(\bar\bz)),
                   &
  \bar\bh(\bar\bz) & \coloneqq  \bh(\bar\bPhi(\bar\bz)).
\end{align}
\Cref{eq:FOM-balanced} represents the balanced model, but the state dimension still matches that of the FOM.
To obtain a ROM, the balancing procedure is followed by a truncation of unimportant states.
The states and model operators in the balanced realization can be partitioned as
\begin{align*}
  \bar\bz = \begin{bmatrix}
              \bar\bz^a \\
              \bar\bz^b
            \end{bmatrix} \in \real^n, \qquad
  \bar\bz_0 = \begin{bmatrix}
                \bar\bz_0^a \\
                \bar\bz_0^b
              \end{bmatrix},  \qquad
  \bar\bf(\bar\bz) = \begin{bmatrix}
                       \bar\bf^a(\bar\bz) \\
                       \bar\bf^b(\bar\bz)
                     \end{bmatrix}, \qquad \bar\bg(\bar\bz) = \begin{bmatrix}
                                                                \bar\bg^a(\bar\bz) \\
                                                                \bar\bg^b(\bar\bz)
                                                              \end{bmatrix},
\end{align*}
after which a choice can be made to truncate states corresponding to small singular value functions $\bar\sigma_i(\bar z_i)$.
The truncated model is given by the dynamics of the leading states $\bar\bz^a \in \real^r$ as
\begin{align}\label{eq:ROM}
  \dot{\bar\bz}^a & = \bar\bf^a\left(\begin{bmatrix}
                                         \bar\bz^a \\
                                         \bzero
                                       \end{bmatrix}\right) + \bar\bg^a\left(\begin{bmatrix}
                                                                               \bar\bz^a \\
                                                                               \bzero
                                                                             \end{bmatrix}\right) \bu(t) ,
                  &
  \by             & = \bar\bh\left(\begin{bmatrix}
                                       \bar\bz^a \\
                                       \bzero
                                     \end{bmatrix}\right),
\end{align}
with the initial condition $\bar\bz^a(0) = \bar\bz_0^a$.
The remaining $n-r$ states were truncated by setting $\bar\bz^b \equiv \bzero$.
\Cref{eq:ROM} represents the ROM, yet the model operators $\bar\bf^a : \real^n \mapsto \real^r$, $\bar\bg^a : \real^n \mapsto \real^{r\times m}$, and $\bar\bh : \real^n \mapsto \real^p$ operate on vectors of the full-order dimension $n$.
Defining $\bx_r \coloneqq \bar\bz^a$,
$\bf_r(\bx_r) \coloneqq \bar\bf^a\left(
  [\bx_r^\top \, \bzero^\top]^\top
  \right)$,
$\bg_r(\bx_r) \coloneqq \bar\bg^a\left(
  [\bx_r^\top \, \bzero^\top]^\top
  \right)$,
$\bh_r(\bx_r) \coloneqq \bar\bh\left(
  [\bx_r^\top \, \bzero^\top]^\top
  \right)$,
and
${\bx_r}_0 \coloneqq \bar\bz_0^a$,
the nonlinear balanced ROM can be written more compactly
in the form \cref{eq:ROM-NL},
an $r$-dimensional approximation to the corresponding $n$-dimensional FOM \cref{eq:FOM-NL}.
The challenge is therefore to obtain the explicit and efficient to evaluate expressions for
the functions $\bf_r:\real^r \mapsto \real^r$, $\bg_r:\real^r \mapsto \real^{r\times m}$, and $\bh_r:\real^r \mapsto \real^p$, and to compute the reduced-order initial condition ${\bx_r}_0 \in \real^r$ in \cref{eq:ROM-NL}.

\subsection{Kronecker product notation}\label{sec:kronecker-preliminaries}
In the interest of developing scalable algorithms, we continue the precedent set by our previous works \cite{Corbin2024a,Corbin2024e,Corbin2025a,Corbin2025b} and related works of \cite{Breiten2018,Borggaard2020,Borggaard2021,Kramer2024} in adopting the Kronecker product representation of polynomials \cite[Sect. 1.3.6 \& 12.3]{Golub2013}.
The Kronecker product of two matrices $\bA \in \real^{p \times q}$ and $\bB \in \real^{s \times t}$ is the $ps \times qt$ block matrix
\begin{align*}
  \bA \otimes \bB \coloneqq \begin{bmatrix} a_{11}\bB & \cdots & a_{1q}\bB \\
                \vdots    & \ddots & \vdots    \\
                a_{p1}\bB & \cdots & a_{pq}\bB
                            \end{bmatrix},
\end{align*}
where $a_{ij}$ denotes the $(i,j)$th entry of $\bA$.
We denote $k$-times repeated Kronecker products as $\kronF{(\cdot)}{k}$,
so that
for a variable $\bx \in \real^n$,
the vector $\kronF{\bx}{k} \coloneqq \bx \otimes \dots \otimes \bx \in \real^{n^k}$
compactly\footnote{The Kronecker approach introduces some redundancy in representing terms like $x_1 x_2 = x_2 x_1$, though these symmetries can be handled efficiently in computations.} represents all monomials of degree $k$.
Kronecker product algebra is a rich topic with a wide range of practical applications in numerical methods; some important properties include $(\bA \otimes \bB)^\top = \bA^\top \otimes \bB^\top$ and $(\bA \otimes \bB)^{-1} = \bA^{-1} \otimes \bB^{-1}$, but notably in general $\bA \otimes \bB \neq \bB \otimes \bA$.
For more details on Kronecker product computations, we refer the reader to \cite[Sect. 1.3.6 \& 12.3]{Golub2013}, and further information can be found in \cite{VanLoan2000,Brewer1978,Henderson1981,Magnus2019}.
Relevant to this work, we include a few special notations based on the Kronecker product, following the convention set in \cite{Borggaard2020,Kramer2023}.
\begin{definition}\label{def:k-way-Lyapunov-matrix}
  For $\bA \in \real^{p \times q}$, the
  \emph{$k$-way Lyapunov matrix} is defined as
  \begin{equation*}
    \cL_k(\bA) \coloneqq \sum_{i=1}^k\underbrace{\bI_p \otimes \bA \otimes \bI_p \otimes \dots \otimes \bI_p}_{\textup{$k$ factors, $\bA$ in the $i$th position}} \in \real^{p^k \times p^{k-1}q}.
  \end{equation*}
\end{definition}
\begin{definition}\label{def:tensor-sum}
  For a set of $\bT_i \in \real^{n \times m^i}$,
  we denote the sum of all unique tensor products with $p$ factors and $m^q$ columns as
  \begin{equation*}
    \cT_{p,q} \coloneqq \sum_{\sum{i_j}=q} \bT_{i_1}\otimes \dots \otimes \bT_{i_p} \in \real^{n^p\times m^q}.
  \end{equation*}
\end{definition}
For example, if the set of coefficients is given by $\bP_i \in \real^{n \times n^i}$, the notation yields $\cP_{3,3} = \bP_1 \otimes \bP_1 \otimes \bP_1 = \kronF{\bP_1}{3} \in \real^{n^3 \times n^3}$, and for a set $\bar\bT_i \in \real^{n \times r^i}$, the notation yields $\bar\cT_{2,3} = \bar\bT_1 \otimes \bar\bT_2 + \bar\bT_2 \otimes \bar\bT_1 \in \real^{n^2 \times r^3}$.

The $k$-way Lyapunov structure often arises when taking derivatives of functions in Kronecker polynomial form, whereas the sum of unique tensor products arises when evaluating the composition of two Kronecker polynomials.
Since the composition of Kronecker polynomials is central to the present work, we present the following Lemma describing the explicit form for the coefficients of the result of composing two Kronecker polynomials,
as worked out in \cite{Corbin2025b}.
\begin{lemma}\label{lem:poly-transformation}
  Let $\bp : \real^n \mapsto \real^n$
  and $\bPhi: \real^r \mapsto \real^n$
  be vector-valued polynomial functions, which we write as
  \begin{align*}
    \bp(\bx)          & =  \sum_{i=1}^\infty \bP_i \kronF{\bx}{i},  &
    \bx  = \bPhi(\bz) & = \sum_{i=1}^\infty \bT_{i} \kronF{\bz}{i},
  \end{align*}
  where $\bP_i \in \real^{n \times n^i}$ and
  $\bT_i \in \real^{n \times r^i}$.
  Then the composition of the two polynomials, $\bp(\bPhi(\bz))$, can be written as
  \begin{align*}
    \bp(\bPhi(\bz)) & =  \sum_{i=1}^\infty \tilde\bP_i \kronF{\bz}{i}
    , \quad \text{where } \quad \tilde\bP_i  =  \sum_{j=1}^i  \bP_j \cT_{j,i},
  \end{align*}
  where
  $\cT_{m,\ell}$
  is the sum of all unique tensor products of $\bT_i$ with $m$ factors and $r^\ell$ columns of \cref{def:tensor-sum}.
\end{lemma}
\begin{proof}
  Substituting $\bx  = \bPhi(\bz)$ into $\bp(\bx)$ yields
  \begin{align*}
     & \bp(\bPhi(\bz))  =  \sum_{i=1}^\infty \bP_i \kronF{(\bPhi(\bz))}{i}
    \\
     & =  \bP_1 (\bT_1 \bz + \bT_2 \kronF{\bz}{2} + \dots)                                                                      +  \bP_2 \left((\bT_1 \bz + \bT_2 \kronF{\bz}{2} + \dots) \otimes (\bT_1 \bz + \bT_2 \kronF{\bz}{2} + \dots )\right)
    + \dots
  \end{align*}
  Collecting the terms of degree $i$, one can see that the coefficient of $\kronF{\bz}{i}$ is given by $\tilde\bP_i  =  \sum_{j=1}^i  \bP_j \cT_{j,i}$.
\end{proof}
Throughout this work, equality will be used in reference to all polynomial expansions of quantities, assuming that the series expansions are taken to infinity.
In practice, when these series expansions are truncated, it is understood that the equality is replaced by an approximation.

\section{Nonlinear balanced truncation via polynomial approximations}\label{sec:polynomial-balancing}
In this section,
we present new mathematical results and algorithms based on Taylor-series expansions to enable scalable, general purpose implementation of the theory of nonlinear balanced truncation.
The proposed approach is based entirely on polynomial approximations, and
an important contribution of the proposed approach is that it permits computing an explicit representation of the balanced realization,
which is used to produce a true ROM.
In appendix \ref{sec:Newton-balancing},
we also describe an alternate way to implicitly \textit{evaluate} the balanced realization using Newton iterations.
In \cref{sec:polynomial-vs-Newton}, we show that the two approaches yield similar results locally;
however, the Newton iteration approach is implicit:
instead of directly evaluating the ROM in \cref{eq:ROM-NL}, the Newton iteration approach evaluates \cref{eq:ROM} via the full-order balanced realization \cref{eq:balanced-realization-operators} as an intermediate step.
This involves projection of the reduced-order state back to the full-order dimension, where the nonlinearity needs to be evaluated, commonly known as a lifting bottleneck.
This issue is exacerbated by the need to perform these evaluations multiple times during each Newton iteration, making the approach significantly more expensive than the proposed polynomial approach.
Nonetheless, it can be used to validate the accuracy of the proposed polynomial approach.

As the starting point,
we assume access to polynomial approximations for the input-normal/output-diagonal transformation
\begin{align}\label{eq:input-normal}
  \bx = \bPhi(\bz) =\bT_1 \bz + \bT_2 \kronF{\bz}{2} + \bT_3 \kronF{\bz}{3} + \dots
\end{align}
and the squared singular value functions
\begin{align}\label{eq:squared-singular-value-functions}
  \sigma^2_i(z_i) = c_0^{(i)} + c_1^{(i)} z_i + c_2^{(i)} z_i^2 + \dots
\end{align}
for $i=1,\dots,n$.
These can be obtained using the methods and associated open-source software described in our previous work~\cite{Corbin2025b}.
In \cref{sec:polynomial-balancing-transformation}, we detail the procedure for obtaining the polynomial expression for the balancing transformation $\bx = \bar\bPhi(\bar\bz)$ given polynomial approximations for the input-normal/output-diagonal transformation
\cref{eq:input-normal}
and the squared singular value functions
\cref{eq:squared-singular-value-functions}.
Subsequently in \cref{sec:polynomial-balanced-realization}, we describe the
procedure to obtain the polynomial expressions for
$\bar\bf(\bar\bz)$,
$\bar\bg(\bar\bz)$, and
$\bar\bh(\bar\bz)$, which form the balanced realization.
Crucially, our method avoids forming and inverting the nonlinear Jacobian $\left[\partial \bar\bPhi(\bar\bz)/\partial \bar\bz\right]^{-1}$, instead only requiring invertibility of the linear component of the balancing transformation.
Then, in \cref{sec:polynomial-balanced-ROM}, we present a new balance-\textit{then}-truncate procedure for obtaining an explicit ROM in polynomial form.

\subsection{Polynomial approximation to the balancing transformation}\label{sec:polynomial-balancing-transformation}
As described in \cref{cor:scaling-1},
the balancing transformation is a composition of two intermediate nonlinear transformations, namely the input-normal/output-diagonal transformation $\bx=\bPhi(\bz)$ and a subsequent scaling transformation $\bz=\bvarphi(\bar\bz)$.
A polynomial expression for the input-normal/output-diagonal transformation is assumed to be given,
but the scaling transformation is not yet known explicitly.
As described in \cref{cor:scaling-1}, the
scaling transformation and its inverse can be summarized as
\begin{align*}
  \bar\bz & = \bvarphi^{-1}(\bz) = \bz \odot \sqrt{\bsigma(\bz)}                  & \text{and} &  &
  \bz     & = \bvarphi(\bar\bz) = \bar\bz \odot \sqrt{\bar\bsigma(\bar\bz)}^{-1},
\end{align*}
where $\odot$ is the Hadamard product,
$\sqrt{\bsigma(\bz)}$ is the vector whose elements are $\sqrt{\sigma_i (z_i)}$,
$\bsigma(\bz) = \bar\bsigma(\bar\bz)$, and
$\sqrt{\bar\bsigma(\bar\bz)}^{-1}$
is the vector whose elements are
$\sqrt{\bar\sigma_i (\bar z_i)}^{-1}$.
To obtain the balancing transformation,
we must first compute a polynomial approximation to the scaling transformation $\bvarphi(\bar\bz)$ using the polynomial expression for the \textit{squared} singular value functions $\bsigma^2(\bz)$.
The main issue is the subtle detail that we have access to the \textit{squared} singular value functions $\bsigma^2(\bz)$ in the input-normal/output-diagonal coordinate system rather than the singular value functions $\bar\bsigma(\bar\bz)$ themselves in the balanced coordinates.
Furthermore, despite the fact that $\bsigma(\bz) = \bar\bsigma(\bar\bz)$, the polynomial expansions for the two are not the same.

Let the $i$th squared singular value function be given by the expansion
\cref{eq:squared-singular-value-functions}.
Then, from \cref{cor:scaling-1}, the $i$th component of the inverse scaling transformation is
\begin{align*}
  \varphi_i^{-1}(z_i) & = z_i \cdot \left( c_0^{(i)} + c_1^{(i)} z_i + c_2^{(i)} z_i^2 + \dots  \right)^{1/4}.
\end{align*}
Assuming that $c_0^{(i)} \neq 0$, as is required in nonlinear balancing,
this expression is analytic, so $\varphi^{-1}(z_i)$ can be locally approximated by a convergent Taylor expansion
\begin{align}
  \varphi_i^{-1}(z_i) & = a_1^{(i)} z_i + a_2^{(i)} z_i^2 + a_3^{(i)} z_i^3 + \dots
\end{align}
where the coefficients are given by the formula
\begin{align} \label{eq:inverse-scaling-coefficients}
  a_k^{(i)} & = \frac{1}{k!} \frac{\rd^k \varphi_i^{-1}(z_i)}{\rd z_i^k } .
\end{align}
The coefficients \cref{eq:inverse-scaling-coefficients} can be pre-computed offline using symbolic computations.
The first five coefficients are
\begin{align*}
  {a_1^{(i)}} & = {c_1^{(i)}}^{1/4},                                                                                                                                      \qquad \qquad\qquad\qquad\quad
  {a_2^{(i)}}  = \frac{{c_2^{(i)}}}{4 {c_1^{(i)}}^{3/4}},                                                                                                                        \qquad\qquad
  {a_3^{(i)}}  = \frac{-3 {c_2^{(i)}}^2 + 8 {c_1^{(i)}} {c_3^{(i)}}}{32 {c_1^{(i)}}^{7/4}},                                                                                                                                    \\
  {a_4^{(i)}} & = \frac{7 {c_2^{(i)}}^3 - 24 {c_1^{(i)}} {c_2^{(i)}} {c_3^{(i)}} + 32 {c_1^{(i)}}^2 {c_4^{(i)}}}{128 {c_1^{(i)}}^{11/4}},                                                                                      \\
  {a_5^{(i)}} & = \frac{-77 {c_2^{(i)}}^4 + 336 {c_1^{(i)}} {c_2^{(i)}}^2 {c_3^{(i)}} - 384 {c_1^{(i)}}^2 {c_2^{(i)}} {c_4^{(i)}} + 64 {c_1^{(i)}}^2 (-3 {c_3^{(i)}}^2 + 8 {c_1^{(i)}} {c_5^{(i)}})}{2048 {c_1^{(i)}}^{15/4}}.
\end{align*}
Given
the expansion for the inverse scaling transformation,
we can construct an expansion for the scaling transformation as
\begin{align*}
  \varphi_i(\bar z_i) = A_1^{(i)} \bar z_i + A_2^{(i)} \bar z_i^2 + A_3^{(i)} \bar z_i^3 + \dots                                                     \end{align*}
using series reversion \cite{Abramowitz2013,Weisstein2025}.
The coefficients can again be pre-computed offline using symbolic computations, with
the first five coefficients being
\begin{align*}
  {A_1^{(i)}} & =\frac{1}{{a_1^{(i)}}},                                                                                                      \qquad\qquad\qquad\qquad\quad
  {A_2^{(i)}}                =-\frac{{a_2^{(i)}}}{{a_1^{(i)}}^3},                                                                                                \qquad\qquad
  {A_3^{(i)}}                =\frac{2{a_2^{(i)}}^2-{a_1^{(i)}}{a_3^{(i)}}}{{a_1^{(i)}}^5},                                                                                              \\
  {A_4^{(i)}} & =-\frac{5{a_2^{(i)}}^3-5{a_1^{(i)}}{a_3^{(i)}}{a_2^{(i)}}+{a_1^{(i)}}^2{a_4^{(i)}}}{{a_1^{(i)}}^7},                                                                     \\
  {A_5^{(i)}} & =\frac{14{a_2^{(i)}}^4-21{a_1^{(i)}}{a_3^{(i)}}{a_2^{(i)}}^2+6{a_1^{(i)}}^2{a_4^{(i)}}{a_2^{(i)}}+3{a_1^{(i)}}^2{a_3^{(i)}}^2-{a_1^{(i)}}^3{a_5^{(i)}}}{{a_1^{(i)}}^9}.
\end{align*}
With these coefficients, a polynomial expansion for the scaling transformation can be constructed in Kronecker product form as
\begin{align}
  \bvarphi(\bar \bz) & = \bA_1 \bar \bz
  + \bA_2 \kronF{\bar \bz}{2}
  + \bA_3 \kronF{\bar \bz}{3}
  + \dots,
  \label{eq:scaling-transformation}
\end{align}
where $\bA_k \in \real^{n \times n^k}$ are sparse polynomial coefficient matrices with exactly $n$ nonzero components given by
\begin{align}
  \bA_k(i,j) = \begin{cases}
                 A^{(m)}_k, & \text{if } i=m \text{ and } j=\frac{m-1}{n-1}(n^k-1) + 1 \\
                 0,         & \text{otherwise}
               \end{cases}.
\end{align}

Upon carrying out the series reversion procedure,
we have the two portions of the balancing transformation in polynomial form: the input-normal/output-diagonal transformation $\bx = \bPhi(\bz)$, and the scaling transformation $\bz = \bvarphi(\bar\bz)$.
\Cref{thm:poly-balancing-transformation} shows how the polynomial approximation to the balancing transformation is obtained from these two polynomial approximations.
\begin{proposition}\label{thm:poly-balancing-transformation}
  Let the input-normal/output-diagonal and scaling transformations for the nonlinear system be expanded in polynomial form
  as in \cref{eq:input-normal,eq:scaling-transformation}
  with coefficients $\bT_k,\bA_k \in \real^{n \times n^k}$ for $k=1,\dots,\infty$, respectively.
  Then the balancing transformation is locally approximated as
  \begin{align}
    \bx = \bar\bPhi(\bar\bz)= \bPhi(\bvarphi(\bar\bz)) = \bar\bT_1 \bar\bz + \bar\bT_2 \kronF{\bar\bz}{2} + \bar\bT_3 \kronF{\bar\bz}{3} + \dots, \label{eq:balancing-transformation}
  \end{align}
  where the balancing transformation coefficients are
  \begin{align}
    \bar\bT_i & = \sum_{j=1}^i  \bT_j \cA_{j,i}, \label{eq:balancing-transformation-coefficients}
  \end{align}
  and
  $\cA_{m,\ell}$
  denotes the sum of all unique tensor products of $\bA_k$ with $m$ factors and $n^\ell$ columns of \cref{def:tensor-sum}.
\end{proposition}
\begin{proof}
  Since $\bx = \bar\bPhi(\bar\bz) = \bPhi(\bvarphi(\bar\bz))$,
  the result follows from the composition of polynomials.
  Application of \cref{lem:poly-transformation} leads to the expression \cref{eq:balancing-transformation-coefficients}.
\end{proof}

\subsection{Polynomial approximation of the dynamics in the balanced realization}\label{sec:polynomial-balanced-realization}
With the explicit polynomial expression for the balancing transformation~\cref{eq:balancing-transformation},
we proceed to
derive expansions for
$\bar\bf(\bar\bz)$,
$\bar\bg(\bar\bz)$, and
$\bar\bh(\bar\bz)$, which are the drift, input, and output maps in the balanced realization \cref{eq:FOM-balanced}.
\begin{theorem}\label{thm:poly-balanced-realization}
  Let the nonlinear system \cref{eq:FOM-NL} be in polynomial form (or be approximated as such) with drift, input, and output maps given by
  \begin{align*}
    \bf(\bx)
     & = \sum_{p=1}^\infty  \bF_p \kronF{\bx}{p},                             &
    \bg(\bx)
     & =  \sum_{p=0}^\infty  \bG_p \left(\bI_m \otimes \kronF{\bx}{p}\right), &
    \bh(\bx)
     & = \sum_{p=1}^\infty  \bH_p \kronF{\bx}{p},
  \end{align*}
  and let the
  nonlinear balancing transformation be approximated in polynomial form as in \cref{eq:balancing-transformation}
  such that in the $\bar{\bz}$ coordinates, the balanced model realization is
  \cref{eq:FOM-balanced}.
  Then a local polynomial approximation to the balanced realization is given by
  \begin{align*}
    \bar\bf(\bar\bz)
     & = \sum_{p=1}^\infty  \bar\bF_p \kronF{\bar\bz}{p},                             &
    \bar\bg(\bar\bz)
     & =  \sum_{p=0}^\infty  \bar\bG_p \left(\bI_m \otimes \kronF{\bar\bz}{p}\right), &
    \bar\bh(\bar\bz)
     & = \sum_{p=1}^\infty  \bar\bH_p \kronF{\bar\bz}{p}.
  \end{align*}
  The explicit formula for the drift coefficients of the balanced realization is
  \begin{align}
    \bar\bF_k
    =  \bar\bT_1^{-1} \left[ \sum_{j=1}^k  \bF_j \bar\cT_{j,k}  - \sum_{i=2}^k
      \bar\bT_i \cL_i \left(\bar\bF_{k-i+1}\right) \right], \label{eq:poly-balanced-drift}
  \end{align}
  where $\cL_i(\cdot)$ is the $i$-way Lyapunov matrix of \cref{def:k-way-Lyapunov-matrix} and $\bar\cT_{j,k}$ is the sum of all unique tensor products of $\bar\bT_i \in \real^{n \times n^i}$ with $j$ factors and $n^k$ columns of \cref{def:tensor-sum}.
  The explicit formula for the input map coefficients of the balanced realization is given by
  $\bar\bG_k = [\bar\bG^{(1)}_k \,\, \bar\bG^{(2)}_k \,\, \dots \,\, \bar\bG^{(m)}_k]$,
  where
  \begin{align*}
    \bar\bG^{(\ell)}_k
    =  \bar\bT_1^{-1} \left[ \sum_{j=1}^k  \bG^{(\ell)}_j \bar\cT_{j,k}  - \sum_{i=2}^k
      \bar\bT_i \cL_i \left(\bar\bG^{(\ell)}_{k-i+1}\right) \right].
  \end{align*}
  The explicit formula for the output coefficients of the balanced realization is
  \begin{align}
    \bar\bH_k
    =   \sum_{j=1}^k  \bH_j \bar\cT_{j,k}. \label{eq:poly-balanced-output}
  \end{align}

\end{theorem}
\begin{proof}
  Let us begin with the proof of the output map coefficients.
  The output map in the balanced coordinates is
  \begin{align*}
    \by = \bh(\bar\bPhi(\bar\bz)),
  \end{align*}
  which is the composition of
  $\bh(\bx) = \sum_{p=1}^\infty  \bH_p \kronF{\bx}{p}$
  and
  $\bar\bPhi(\bar\bz) = \sum_{p=1}^\infty  \bar\bT_p \kronF{\bar\bz}{p}$.
  The result
  can be computed straightforwardly
  using
  \cref{lem:poly-transformation}, yielding the formula in \cref{eq:poly-balanced-output}.

  The transformed state equation in the balanced coordinates is given by
  \begin{equation}
    \frac{\partial \bar\bPhi(\bar\bz)}{\partial \bar\bz} \dot{\bar\bz} = \bf(\bar\bPhi(\bar\bz)) + \bg(\bar\bPhi(\bar\bz)) \bu(t). \label{eq:4}
  \end{equation}
  Inserting
  \cref{eq:FOM-balanced} into \cref{eq:4} gives
  \begin{align}
    \frac{\partial \bar\bPhi(\bar\bz)}{\partial \bar\bz} \left(\bar\bf(\bar\bz) + \bar\bg(\bar\bz) \bu(t)\right)
     & = \bf(\bar\bPhi(\bar\bz)) + \bg(\bar\bPhi(\bar\bz))\bu(t) , \label{eq:5}
  \end{align}
  and rather than invert the nonlinear Jacobian, we proceed to insert the polynomial expressions for each of these quantities and match coefficients for terms of the same polynomial degree on the left and right sides.

  Starting with the drift terms, we seek to match the degree $k$ terms on the left and right of \cref{eq:5}.
  On the right, the quantity
  $\bf(\bar\bPhi(\bar\bz))$ is a composition of two polynomials, and by \cref{lem:poly-transformation}, the $k$th coefficient can be expressed as
  $\sum_{j=1}^k \bF_j \bar\cT_{j,k}$.
  On the left,
  the Jacobian of the polynomial balancing transformation
  can be written without loss of generality as
  \begin{align*}
    \frac{\partial \bar\bPhi(\bar\bz)}{\partial \bz}
     & = \sum_{p=1}^\infty p\bar\bT_p \left( \bI_n \otimes \kronF{\bar\bz}{p-1}\right).
  \end{align*}
  Inserting these expressions, the drift terms are matched as
  \begin{align}
    \sum_{i=1}^\infty i\bar\bT_i \left( \bI_n \otimes \kronF{\bar\bz}{i-1}\right)
    \sum_{j=1}^\infty  \bar\bF_j \kronF{\bar\bz}{j} & = \sum_{k=1}^\infty  \sum_{j=1}^k \bF_j \bar\cT_{j,k} \kronF{\bar\bz}{k}. \label{eq:20}
  \end{align}
  From here, the terms on the left must be manipulated
  to match the terms on the right.
  Using Kronecker product identities, it is possible to write an arbitrary degree~$k$ term on the left using the $i$-way Lyapunov matrix of \cref{def:k-way-Lyapunov-matrix} as
  \begin{align*}
    i\bar\bT_i \left(\bI_n \otimes \kronF{\bar\bz}{i-1} \right)
    \bar\bF_j \kronF{\bar\bz}{j}
     & =  \bar\bT_i \cL_i(\bar\bF_j) \kronF{\bar\bz}{k}
  \end{align*}
  for $i+j-1=k$.
  With this substitution,
  the collection of degree $k$ terms on left and right in \cref{eq:20} can be rewritten as
  \begin{align*}
    \bar\bT_1\bar\bF_k \kronF{\bar\bz}{k} + \sum_{i=2}^k
    \bar\bT_i \cL_i \left(\bar\bF_{k-i+1}\right) \kronF{\bar\bz}{k}
    =   \sum_{j=1}^k \bF_j \bar\cT_{j,k} \kronF{\bar\bz}{k},
  \end{align*}
  where the unknown coefficient $\bar\bF_k$ has been isolated from the other terms.
  Rearranging only requires the inversion of $\bar\bT_1$, whose inverse is analytically known, yielding the expression in \cref{eq:poly-balanced-drift}.

  For the input map coefficients, we consider each input separately.
  The state equation of the dynamics can be rewritten as
  \begin{align*}
    \dot{\bx}
     & = \bf(\bx) + \sum_{i=1}^{m} \bg^{(i)}(\bx) u_i(t),
  \end{align*}
  where $\bg^{(i)}(\bx) =  \sum_{p=0}^\infty  \bG^{(i)}_p  \kronF{\bx}{p} \in \real^n$ denotes the $i$th column of the matrix
  $\bg(\bx)=  \sum_{p=0}^\infty  \bG_p \left(\bI_m \otimes \kronF{\bx}{p}\right) \in \real^{n\times m}$.
  The coefficients of $\bg(\bx)$ and $\bg^{(i)}(\bx)$ are related
  by
  $\bG_p = [ \bG^{(1)}_p \,\, \bG^{(2)}_p \,\, \dots \,\, \bG^{(m)}_p ]$.
  Concerning the $i$th input $u_i(t)$, the same procedure as was used for the drift terms can be used.
  Writing the balanced realization as
  \begin{align*}
    \dot{\bar\bz} & = \bar\bf(\bar\bz) + \sum_{i=1}^{m} \bar\bg^{(i)}(\bar\bz) u_i(t),
  \end{align*}
  the polynomial expansion for the $i$th input vector field is
  \begin{align*}
    \bar\bg^{(i)}(\bar\bz) & = \sum_{p=1}^\infty  \bar\bG^{(i)}_p \kronF{\bar\bz}{p},
                           & \bar\bG^{(\ell)}_k
                           & =  \bar\bT_1^{-1} \left[ \sum_{j=1}^k  \bG^{(\ell)}_j \bar\cT_{j,k}  - \sum_{i=2}^k
      \bar\bT_i \cL_i \left(\bar\bG^{(\ell)}_{k-i+1}\right) \right].
  \end{align*}
  To write the dynamics in the original control-affine form, the individual input vector field polynomial coefficients are stacked in block form as
  $\bar\bG_k = [\bar\bG^{(1)}_k \,\, \bar\bG^{(2)}_k \,\, \dots \,\, \bar\bG^{(m)}_k]$.
\end{proof}

The final requirement for the balanced realization is the transformed initial condition,
which is provided by the next theorem.

\begin{theorem}\label{thm:poly-inverse}
  Let $\bx(0) = \bx_0$ be the initial condition of the control-affine dynamical system \cref{eq:FOM-NL},
  and let the system's nonlinear balancing transformation be approximated in polynomial form as in \cref{eq:balancing-transformation}.
  Then the initial condition in the balanced realization is $\bar\bz(0) = \bar\bz_0 = \bar\bPhi^{-1}(\bx_0)$, where the inverse balancing transformation is locally approximated by
  \begin{align}\label{eq:inverse-balancing-transformation}
    \bar\bz = \bar\bPhi^{-1}(\bx)
     & = \bP_1 \bx + \bP_2 \kronF{\bx}{2} + \bP_3 \kronF{\bx}{3} +\dots
  \end{align}
  where $\bP_1 = \bar\bT_1^{-1}$ and the remaining coefficients $\bP_i$ for $i \geq 2$ are given by
  \begin{align*}
    \bP_i = \left(-\sum_{j=1}^{i-1} \bP_j \bar\cT_{j,i}\right)
    \kronF{\bP_1}{i},
  \end{align*}
  where $\bar\cT_{j,i}$ is the sum of all unique tensor products of $\bar\bT_k \in \real^{n \times n^k}$ with $j$ factors and $n^i$ columns of \cref{def:tensor-sum}.
\end{theorem}
\begin{proof}
  The inverse transformation locally satisfies
  $\bar\bPhi^{-1}(\bar\bPhi(\bar\bz)) = \bar\bz$, which is recognized again as the composition of two polynomials.
  Therefore, applying \cref{lem:poly-transformation} and then equating terms of the same polynomial degree yields equations for each inverse transformation coefficient $\bP_i$.
  The first coefficient equation shows that $\bP_1$ is the inverse of the linear balancing coefficient $\bar\bT_1$:
  \begin{align*}
    \bP_1 \bar\bT_1 \bar\bz =  \bar\bz \qquad \to \qquad \bP_1 = \bar\bT_1^{-1}.
  \end{align*}
  The remaining equations for $\bP_i$ can be derived using \cref{lem:poly-transformation} as
  \begin{align*}
    \sum_{j=1}^i \bP_j \bar\cT_{j,i} = \sum_{j=1}^{i-1} \bP_j \bar\cT_{j,i} + \bP_i \bar\cT_{i,i} & = 0  \qquad \to \qquad
    \bP_i = \left(-\sum_{j=1}^{i-1} \bP_j \bar\cT_{j,i}\right)\bar\cT_{i,i}^{-1}.
  \end{align*}
  According to the properties of the Kronecker product,  $\bar\cT_{i,i}^{-1} = (\bar\bT_1 \otimes \dots \otimes \bar\bT_1)^{-1}$ can be expanded and replaced with $(\bar\bT_1^{-1} \otimes \dots \otimes \bar\bT_1^{-1}) = (\bP_1 \otimes \dots \otimes \bP_1) =
    \kronF{\bP_1}{i}$,
  completing the proof.
\end{proof}
\begin{remark}
  The linear balancing transformation coefficient $\bar\bT_1$ is typically obtained using the square-root balancing formulation \cite[Sect. 7.3]{Antoulas2005}, in which case its inverse is analytically known.
  Therefore, $\bP_1 = \bar\bT_1^{-1}$ can be obtained without needing to explicitly compute the inverse of $\bar\bT_1$.
\end{remark}

\subsection{Model order reduction by truncation}\label{sec:polynomial-balanced-ROM}
\Cref{thm:poly-balanced-realization,thm:poly-inverse} describe how to transform the original dynamics in \cref{eq:FOM-NL} into a balanced realization \cref{eq:FOM-balanced} using polynomial approximations of the balancing transformation \cref{eq:balancing-transformation}.
In order to obtain the ROM \cref{eq:ROM-NL},
we truncate the trailing $n-r$ states, leading to
a balance-then-reduce transformation of the form
$\bar\bPhi_r : \real^r \mapsto \real^n$.
If the full-order balancing transformation is given in polynomial form \cref{eq:balancing-transformation}
with coefficients $\bar\bT_k \in \real^{n \times n^k}$,
then the reduced-order transformation can be written to degree $d_\text{transf}$ as
\begin{align}
  \bar\bPhi_r(\bx_r) & \approx \bT^{(r)}_1 \bx_r + \bT^{(r)}_2 \kronF{\bx_r}{2} + \dots  +  \bT^{(r)}_{d_\text{transf}} \kronF{\bx_r}{d_\text{transf}}, \label{eq:rom-transformation}
\end{align}
where the coefficients $\bT^{(r)}_k \in \real^{n\times r^k}$ are obtained by simply eliminating the columns of $\bar\bT_k \in \real^{n\times n^k}$ that correspond to any of the truncated states.
This transformation converts the original dynamics in \cref{eq:FOM-NL} into a degree $d_\text{ROM}$ approximation of the ROM \cref{eq:ROM-NL} given by
$\bf_r(\bx_r) = \sum_{p=1}^{d_\text{ROM}}  \bF^{(r)}_p \kronF{\bx_r}{p}$,
$\bg_r(\bx_r) =  \sum_{p=0}^{d_\text{ROM}-1}  \bG^{(r)}_p \left(\bI_m \otimes \kronF{\bx_r}{p}\right)$,
and
$\bh_r(\bx_r) = \sum_{p=1}^{d_\text{ROM}}  \bH^{(r)}_p \kronF{\bx_r}{p}$.
We typically set $d_\text{ROM} = d_\text{transf}$, but they can also be set independently.
The procedure for obtaining a polynomial ROM using nonlinear balanced truncation is summarized in \cref{alg:alg1}, which is implemented in \textsc{Matlab} in the function \texttt{getBalanceThenReduceRealization()} in the \texttt{NLbalancing} repository \cite{Corbin2025c}.

\begin{algorithm}[h!]
  \caption{Computation of ROMs via nonlinear balanced truncation}\label{alg:alg1}
  \begin{algorithmic}[1]
    \Require
    Polynomial coefficients for $\bf(\bx)$, $\bg(\bx)$, $\bh(\bx)$ in \cref{eq:FOM-NL};
    initial condition $\bx_0$;
    desired transformation degree $d_\text{transf}$,
    ROM degree $d_\text{ROM}$,
    ROM dimension $r$.
    \Ensure
    Polynomial coefficients for $\bf_r(\bx_r)$, $\bg_r(\bx_r)$, $\bh_r(\bx_r)$ in \cref{eq:ROM-NL};
    initial condition $\bx_{r0}$  in \cref{eq:ROM-NL}.
    \State Compute degree $d_\text{transf}+1$ approximations to the controllability and observability energy functions $\cE_c(\bx)$ and $\cE_o(\bx)$ using the methods and open-source software proposed in \cite[Thms. 1 \& 2]{Corbin2025a}.
    \State Compute a degree $d_\text{transf}$ approximation to the input-normal/output-diagonal transformation $\bPhi(\bz)$ and the squared singular value functions $\bsigma^2(\bz)$ using the methods and open-source software proposed in \cite[Thm. 2]{Corbin2025b}.
    \State Compute a degree $d_\text{transf}$ approximation to the full-order balancing transformation $\bar\bPhi(\bar\bz)$ using \cref{thm:poly-balancing-transformation}.
    \State Truncate the full-order balancing transformation $\bar\bPhi(\bar\bz)$ to obtain the balance-\textit{then}-truncate transformation $\bar\bPhi_r(\bx_r)$ in \cref{eq:rom-transformation}.
    \State Compute a degree $d_\text{ROM}$ ROM and initial condition using \cref{thm:poly-balanced-realization,thm:poly-inverse}.
  \end{algorithmic}
\end{algorithm}

We refer to this as a balance-\textit{then}-truncate procedure because the full balancing transformation must be computed before truncation can take place.
Somewhat paradoxically then,
the propagation of small Hankel singular values (which make the system amenable to model order reduction) throughout the calculations can make this approach numerically ill-conditioned, analogous to the linear case \cite[Sect. 7.3]{Antoulas2005}.
The square-root balancing implementation of linear balanced truncation, sometimes referred to as a balance-\textit{and}-truncate, circumvents this ill-conditioning by reformulating the manner in which the transformation is computed to avoid the inversion of small Hankel singular values.

It is not immediately clear how a balance-\textit{and}-truncate procedure for the nonlinear case could be constructed.
We compute the linear transformation coefficient using the square-root balancing algorithm,
but the computation of the higher-order terms in the balancing transformation still requires the inversion of the small Hankel singular values.
One might naively hope that truncating $\bar\bT_1$ immediately to  $\bT_1^{(r)}$ and then carrying on with the rest of the nonlinear balancing computations would produce the desired balance-and-truncate nonlinear transformation, but it does not.
Similarly, we attempted to modify the computations for example in the input-normal/output-diagonal transformation to target only the first $r$ states, but again this does not produce the same result as the proper balance-then-truncate transformation.

\section{Numerical examples}\label{sec:examples}
We illustrate the proposed algorithms on a variety of examples to demonstrate the characteristics of the balancing transformations and the resulting ROMs.
The results are obtained with a Dell XPS 15 laptop featuring an Intel i7-13700H CPU and 64 GB of RAM using \textsc{Matlab} 2024a.
All examples are provided with the implementation in the \texttt{cnick1/NLbalancing} repository \cite{Corbin2025c}.

The first two examples in \cref{sec:example1,sec:example2}, respectively, illustrate features of the nonlinear balancing transformation without truncation.
The nonlinear balancing transformation is a change of coordinates characterized by a state-dependent rotation and stretching of the coordinate grid.
Therefore, if computed correctly, the differential equation solutions in the original coordinates should match the solutions in the transformed coordinates, otherwise a ROM based on that transformation may not be trustworthy.

In \cref{sec:example3,sec:example4}, we proceed to truncate the balanced realization to illustrate the model reduction procedure, first on an academic 3D model, and then on a double pendulum.
Finally, in \cref{sec:example5}, we demonstrate the scalability of the proposed approach on a finite element nonlinear beam example.
We show that nonlinear balancing ROMs are capable of capturing behaviors that linear balancing ROMs cannot qualitatively capture.
We then show that the Kronecker product-based computations we propose enable scalability into the high hundreds of state dimensions.

\subsection{2D illustrative example}\label{sec:example1}
The first model we consider is an academic example originating in \cite{Gray2001}.
The control-affine dynamics are given by
\begin{align*}
  \bf(\bx) & = -\begin{pmatrix}
                  \alpha^2 x_1 + 2 \alpha x_2 + (\alpha^2 - 2)x_2^2 \\
                  x_2
                \end{pmatrix}    ,                                &
  \bg(\bx) & = \sqrt{2}\left(\begin{matrix}
                                 \alpha - 2 x_2 \\
                                 1
                               \end{matrix}\right),                                                \\
  \bh(\bx) & = \frac{1}{\sqrt{3}} \left(3 \alpha (x_1 + x_2^2) + (\alpha - 2\sqrt{2})x_2\right),
\end{align*}
where $\alpha = \left(\sqrt{3} + \sqrt{2}\right)\left(\sqrt{3} + 2\right)$.
Using our software package, we compute the controllability and observability energy functions as
\begin{align*}
  \cE_c(\bx) & =  \frac{1}{2} (x_1^2 + 2.0 x_1 x_2^2 + x_2^4 + x_2^2) ,                             \\
  \cE_o(\bx) & = \frac{1}{2} (1.5 x_1^2 + x_1 x_2 + 1.5 x_2^2 + 3.0 x_1 x_2^2 + x_2^3 + 1.5 x_2^4).
\end{align*}
The computed energy functions are
consistent with those shown
in \cite{Gray2001}, in which they were shown to be exactly quartic.
The system is placed in input-normal/output-diagonal form by the transformation
\begin{align*}
  \bPhi(\bz) = \begin{bmatrix}
                 - 0.5 z_1^2 + z_1 z_2 - 0.707 z_1 - 0.5 z_2^2 - 0.707 z_2 \\
                 0.707 z_2 - 0.707 z_1
               \end{bmatrix},
\end{align*}
revealing the squared singular value functions to be $\sigma_1^2(z_1) = 2$ and $\sigma_2^2(z_2) =  1$ within $O(10^{-14})$.
Using the singular value functions to compute the scaling transformation, the full nonlinear balancing transformation computed to degree 7 is
\begin{align*}
  \bar\bPhi(\bar\bz) = \begin{bmatrix}
                         - 0.354\bar{z}_1^2 + 0.841\bar{z}_1\bar{z}_2 - 0.595\bar{z}_1 - 0.5\bar{z}_2^2 - 0.707\bar{z}_2 \\
                         0.707\bar{z}_2 - 0.595\bar{z}_1
                       \end{bmatrix},
\end{align*}
which reveals the following balanced realization:
\begin{align*}
  \dot{\bar z}_1 & = - 81.2 \bar{z}_1 - 67.4 \bar{z}_2 -15.2 u(t), \\
  \dot{\bar z}_2 & = - 67.4 \bar{z}_1 - 57.7 \bar{z}_2 -10.7 u(t), \\
  y              & = - 15.2 \bar{z}_1 - 10.7 \bar{z}_2.
\end{align*}
Curiously, the balanced realization is linear, which suggests
that this nonlinear model was generated by applying a quadratic transformation to a linear model.
This is consistent with the singular value functions being constant.
Thus, the balancing transformation \textit{undoes} the nonlinear transformation to recover the linear balanced realization, and in this case the balancing transformation is exactly quadratic.

In \cref{fig:example13_gridTransformations}, we illustrate the action of the balancing coordinate transformation.
Taking a grid in the balanced $\bar\bz$ coordinates, we apply the transformation $\bx = \bar\bPhi(\bar\bz)$ to each point on the grid to map to the original $\bx$ coordinates.
\begin{figure}[htb]
  \centering
  \includegraphics[width=\textwidth]{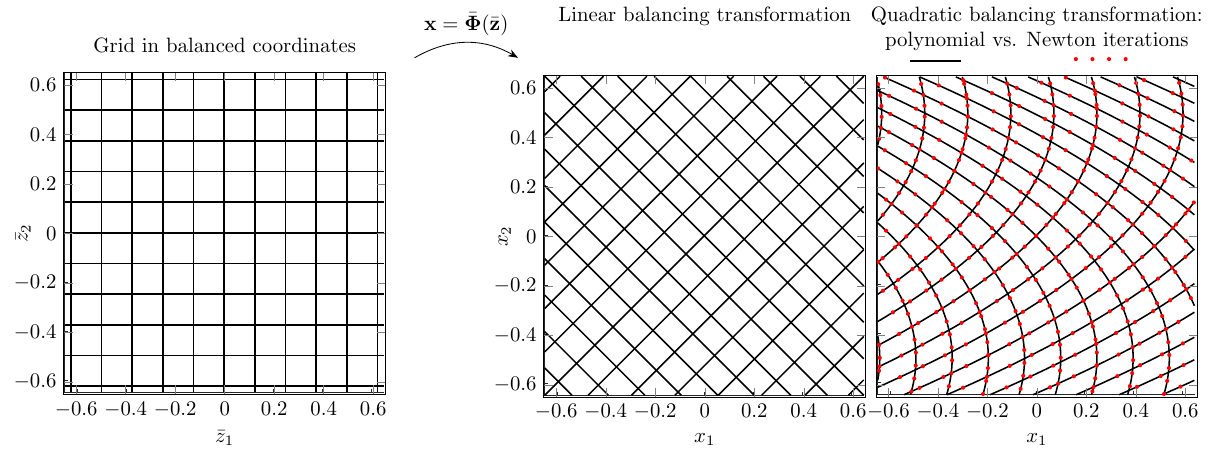}
  \caption{2D illustrative example: transformation of a grid from the balanced coordinates to the original coordinates via $\bx = \bar\bPhi(\bar\bz)$. While the polynomial approach generally introduces additional polynomial approximations compared to the Newton iteration approach, it yields the same degree of accuracy locally. }
  \label{fig:example13_gridTransformations}
\end{figure}
For the linear transformation $\bx = \bar\bT_1 \bar\bz$, the matrix $\bar\bT_1$ represents a rotation and stretching, which keeps grid lines straight.
Under the nonlinear balancing transformation $\bx = \bar\bPhi(\bar\bz)$, the amount of stretching and rotation can change as coordinates get farther from the origin, imparting curvature to the grid lines.
Observe that, locally about the origin, the nonlinear transformation is approximated by the linear transformation.

Along with the proposed polynomial approximations to the balancing transformation, we also include a comparison with the Newton iteration approach to computing the balancing transformation detailed in appendix \ref{sec:Newton-balancing} in the grid transformation plots in \cref{fig:example13_gridTransformations}.
As described in appendix \ref{sec:polynomial-vs-Newton}, both of these approaches build on the polynomial input-normal/output-diagonal transformation proposed in \cite{Corbin2025b} and should exhibit approximately the same level of accuracy.
However, the Newton iteration approach uses various nested Newton iterations to implicitly evaluate the balancing transformation, whereas the polynomial approach outlined in \cref{sec:polynomial-balancing} introduces a few additional steps of polynomial approximations to obtain an explicit approximation of the balancing transformation.
The comparisons in \cref{fig:example13_gridTransformations} confirm that both methods are capable of computing local approximations to the balancing transformation; however, the polynomial approach is much faster, as it does not involve the repeated solution of iterative equations.
Furthermore, the Newton iteration approach retains a lifting-type bottle neck due to its implicit nature, so once truncation of the model order is considered, the fully polynomial approach is the only truly reduced-order approach.
Therefore, in the remaining examples, we will restrict our attention to the fully polynomial approach.

\subsection{2D stable pendulum}\label{sec:example2}
The next example we consider is a forced damped pendulum with torque input and position measurement, as considered in \cite{Newman1999}.
The model is described by the governing equations
\begin{equation}\label{eq:pendulum-FOM}
  \begin{split}
    \dot{x}_1 & = x_2,                                                                   \\
    \dot{x}_2 & = -\frac{G}{L} \sin(x_1) - \frac{k}{mL^2}x_1 - \frac{b}{mL^2}x_2 + u(t), \\
    y         & = x_1,
  \end{split}
\end{equation}
where $G=10$, $L=20$, $m=1/40$, $b=2$, and $k=1$.

We focus in this example specifically on the topic of bijectivity of the balancing transformation \textit{before} truncation of the model dimension.
Prior to truncation, solution trajectories in the balanced realization should map identically to solutions in the original coordinates.
In \cite{Newman1999}, the authors derive analytical solutions for the energy functions for this system, which exhibits special structure, and they proceed to use a discretization-based algorithm to compute the balancing transformation.
Their results highlight that, even with analytical solutions for the energy functions, it is a significant challenge to compute a well-behaved balancing transformation for which solution trajectories in the original coordinate system exactly match the solution trajectories in the computed balanced coordinates.
Errors introduced by their approach lead to visible mismatches between the original solution trajectories and the transformed model's solution trajectories.

We begin by computing the energy functions, shown in \cref{fig:example31_energyFunctions} along with their residuals as an error metric.
For the region of the state-space shown with $-1 \leq x_i \leq 1$, a very good fit can be obtained with degree 4 approximations to the energy functions.

\begin{figure}[ht!]
  \centering
  \begin{subfigure}[b]{0.48\textwidth}
    \centering
    \includegraphics[width=.49\linewidth]{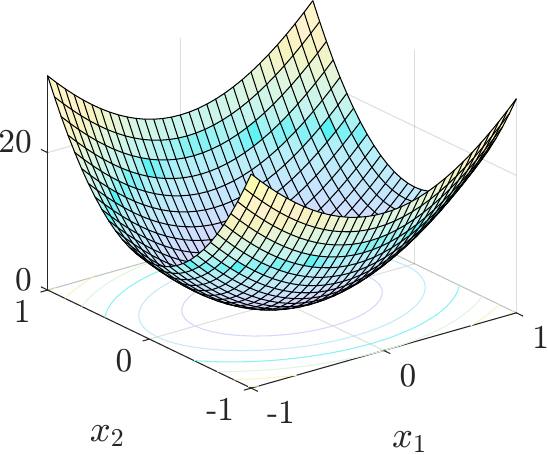}
    \includegraphics[width=.49\linewidth]{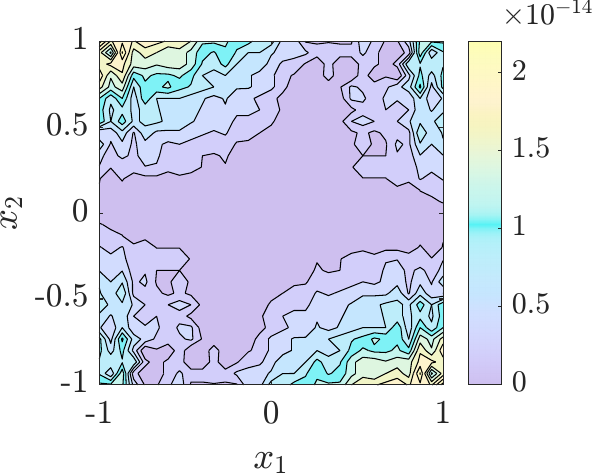}
    \caption{Controllability energy function and its residual.}
    \label{fig:example31_ctrb_res}
  \end{subfigure}
  \hfill
  \begin{subfigure}[b]{0.48\textwidth}
    \centering
    \includegraphics[width=.49\linewidth]{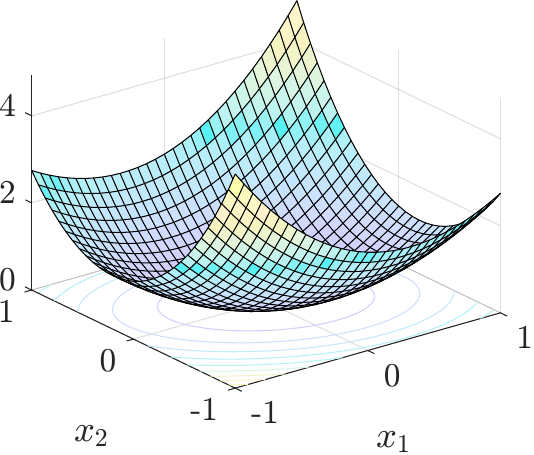}
    \includegraphics[width=.49\linewidth]{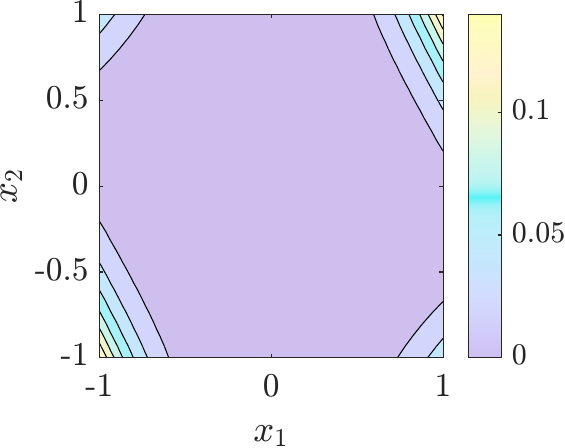}
    \caption{Observability energy function and its residual.}
    \label{fig:example31_obsv_res}
  \end{subfigure}
  \caption{2D stable pendulum: degree 4 energy function approximations and their residuals.}
  \label{fig:example31_energyFunctions}
\end{figure}

The first test case considered in \cite{Newman1999} is the system's response to the initial condition $\bx(0) = [0.1, 0.1]^\top$ with zero input.
We compare the full-order, fully nonlinear response with a cubic approximation to the balanced realization
in \cref{fig:example31_impulse}.
In \cite{Newman1999}, the solutions in the balanced and original coordinates showed some mismatch, whereas our transformation better retains equivalence of the solutions.
It is worth noting that for this relatively small initial condition, the model operates essentially in its linear regime.
This can be seen by observing that the coordinate grid overlaid in the phase-space plots rotates and stretches, but does not warp in a nonlinear fashion.
Indeed, since we do not consider any truncation at the moment, even the linear balancing transformation would work well here, as its invertibility guarantees its bijectivity.

\begin{figure}[ht!]
  \centering
  \begin{subfigure}[b]{0.44\textwidth}
    \centering
    \includegraphics[width=\linewidth]{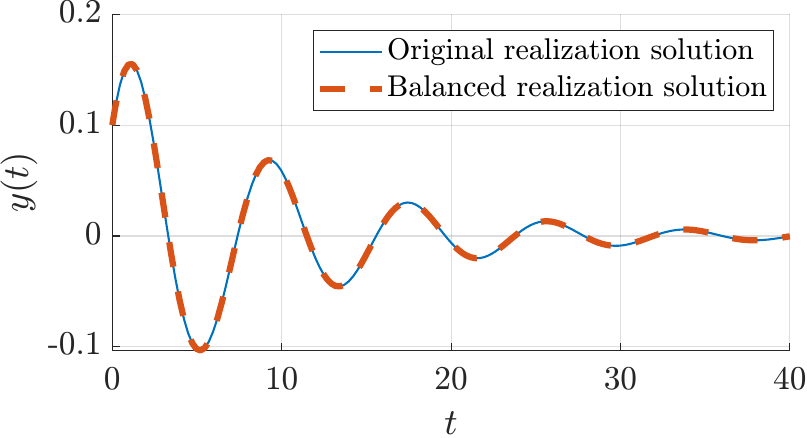}
    \caption{Angular position output $y(t) = x_1(t)$.\\\,}
    \label{sfig:example31_impulse}
  \end{subfigure}
  \hfill
  \begin{subfigure}[b]{0.5\textwidth}
    \centering
    \includegraphics[width=0.49\linewidth]{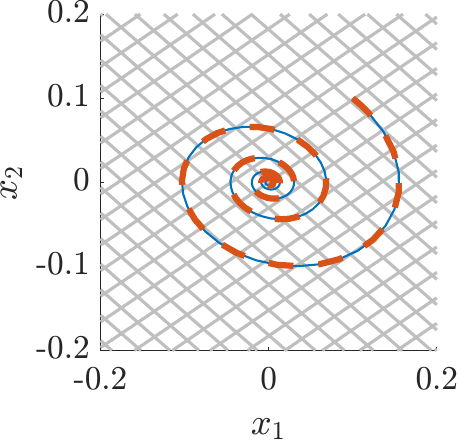}
    \includegraphics[width=0.49\linewidth]{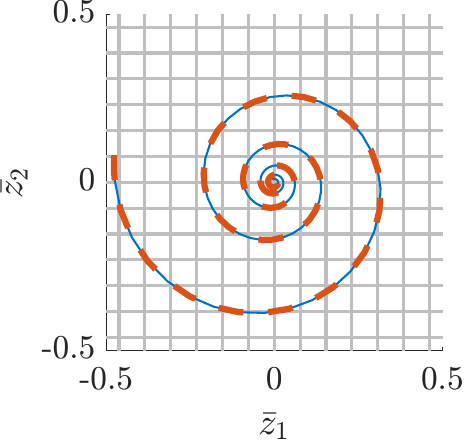}
    \caption{Solution state trajectory in the original (left) and balanced (right) coordinates.}
    \label{fig:example31_impulse_z}
  \end{subfigure}
  \caption{2D stable pendulum: unforced response for initial condition $\bx(0) = [0.1, 0.1]^\top$ with a degree 3 balanced realization approximation.}
  \label{fig:example31_impulse}
\end{figure}

The second test case considered is the response to a sinusoidal forcing $u(t) = 0.5 \sin(t/\pi)$ and the same nonzero initial condition, shown in \cref{fig:example31_sin1}.
Similar to the unforced case, our method maintains bijectivity of the balancing transformation, which again is perhaps unsurprising since the response remains in the linear regime.
Still, the results shown in \cite{Newman1999} show a small error, indicating the difficulty of computing a nonlinear balancing transformation without introducing additional error.

\begin{figure}[ht!]
  \centering
  \begin{subfigure}[b]{0.44\textwidth}
    \centering
    \includegraphics[width=\linewidth]{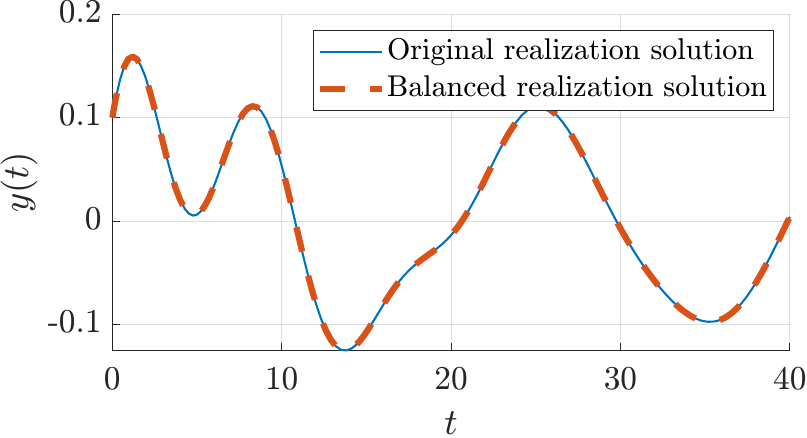}
    \caption{Angular position output $y(t) = x_1(t)$.\\\,}
    \label{sfig:example31_sin1}
  \end{subfigure}
  \hfill
  \begin{subfigure}[b]{0.5\textwidth}
    \centering
    \includegraphics[width=0.49\linewidth]{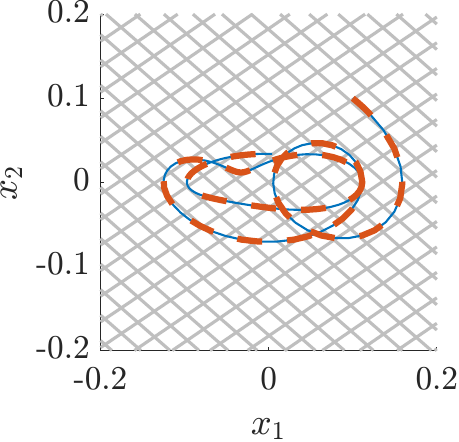}
    \includegraphics[width=0.49\linewidth]{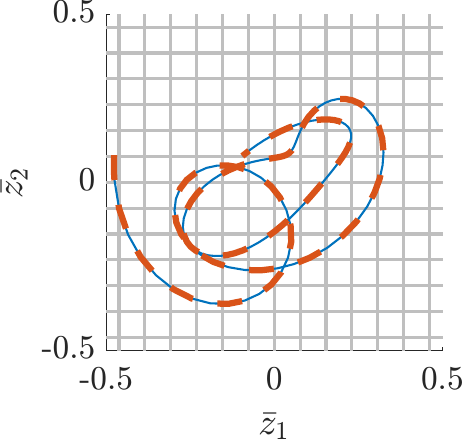}
    \caption{Solution state trajectory in the original (left) and balanced (right) coordinates.}
    \label{fig:example31_sin1_z}
  \end{subfigure}
  \caption{2D stable pendulum: response to sinusoidal input $u(t) = 0.5 \sin(t/\pi)$ for initial condition $\bx(0) = [0.1, 0.1]^\top$ with a degree 3 balanced realization approximation.}
  \label{fig:example31_sin1}
\end{figure}

As a final test, noting that the previous two test cases from \cite{Newman1999} remained in the linear regime, we increase the initial condition to $\bx(0) = [1, 1]^\top$ and the forcing to $u(t) = 5 \sin(t/\pi)$, as shown in \cref{fig:example31_sin2_d4}.
The reader should take note that the axis limits now reflect a larger portion of the state space.
The nonlinearities actually contribute over this region, as illustrated by the curvature present in the coordinate grids overlaid on the phase-space plots.
Now, we can see that the state components stray outside of the $-1 \leq x_i \leq 1$ range in which we had good energy function approximations, and correspondingly we have some error in the state trajectory.
The source of this error is the polynomial truncation (originally in the dynamics, and then in the energy functions and transformations), as the cubic approximation is failing to capture the full behavior of $\sin(x_1)$ on this domain.
The issue can be resolved by computing a higher-order approximation to the balanced realization.
Instead of the cubic balancing transformation, the results in \cref{fig:example31_sin2_d8} feature a degree 7 balancing transformation, and we see that the accuracy of the solution in the balanced coordinate system is recovered.

\begin{figure}[ht!]
  \centering
  \begin{subfigure}[b]{0.44\textwidth}
    \centering
    \includegraphics[width=\linewidth]{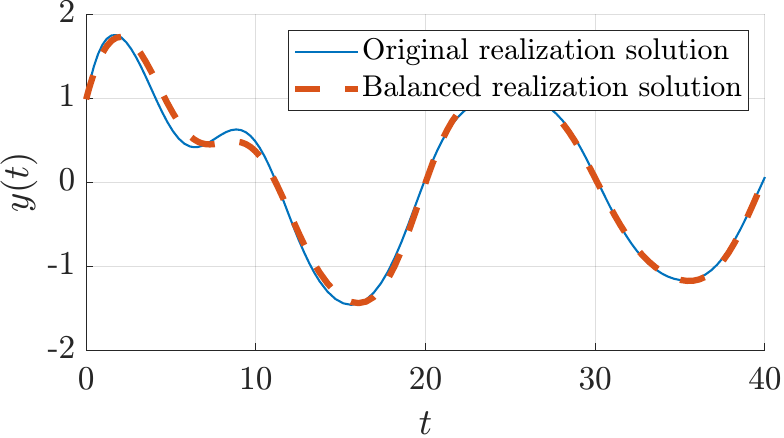}
    \caption{Angular position output $y(t) = x_1(t)$.\\\,}
    \label{sfig:example31_sin2_d4}
  \end{subfigure}
  \hfill
  \begin{subfigure}[b]{0.5\textwidth}
    \centering
    \includegraphics[width=0.45\linewidth]{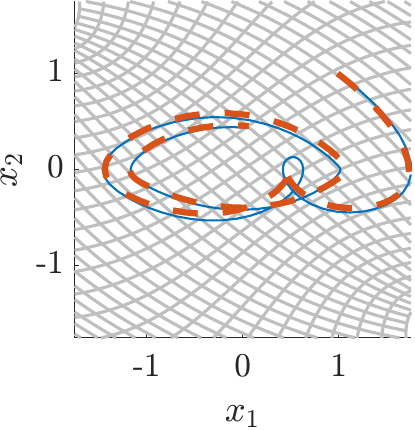}
    \includegraphics[width=0.45\linewidth]{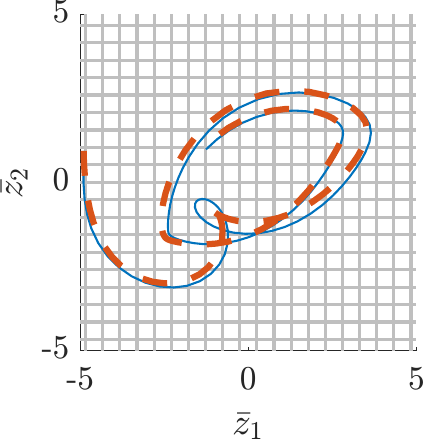}
    \caption{Solution state trajectory in the original (left) and balanced (right) coordinates.}
    \label{fig:example31_sin2_z_d4}
  \end{subfigure}
  \caption{2D stable pendulum: response to sinusoidal input $u(t) = 5 \sin(t/\pi)$ for initial condition $\bx(0) = [1, 1]^\top$ with a degree 3 balanced realization approximation.
  The cubic approximation fails to capture the full nonlinearity $\sin(x)$, leading to the mismatch between the trajectories.}
  \label{fig:example31_sin2_d4}
\end{figure}

\begin{figure}[ht!]
  \centering
  \begin{subfigure}[b]{0.44\textwidth}
    \centering
    \includegraphics[width=\linewidth]{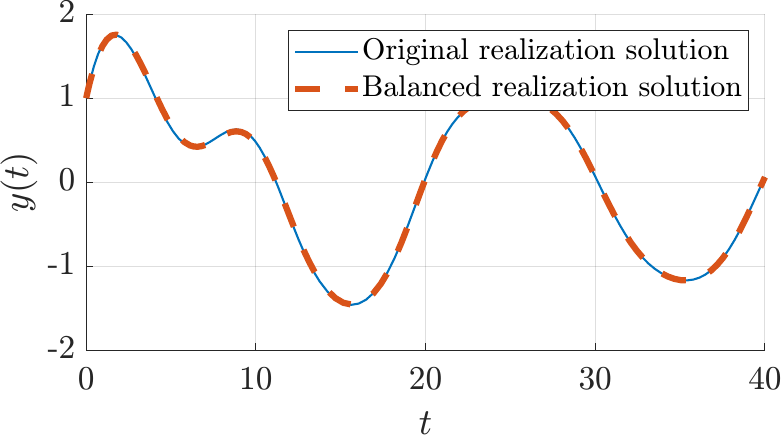}
    \caption{Angular position output $y(t) = x_1(t)$.\\\,}
    \label{sfig:example31_sin2_d8}
  \end{subfigure}
  \hfill
  \begin{subfigure}[b]{0.5\textwidth}
    \centering
    \includegraphics[width=0.45\linewidth]{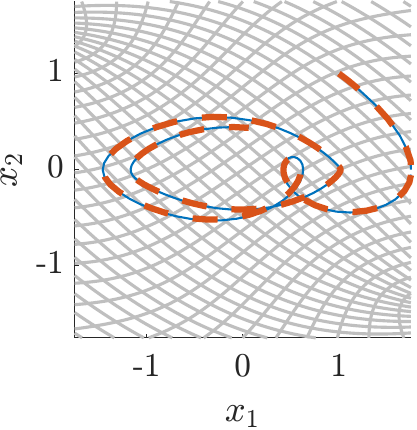}
    \includegraphics[width=0.45\linewidth]{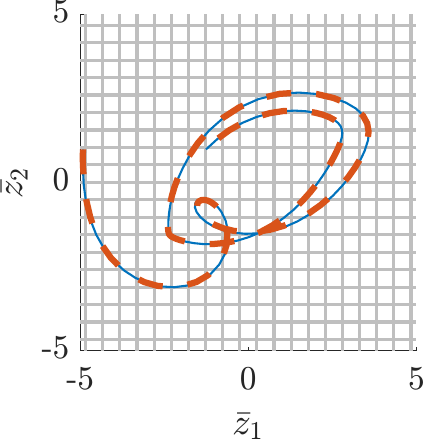}
    \caption{Solution state trajectory in the original (left) and balanced (right) coordinates.}
    \label{fig:example31_sin2_z_d8}
  \end{subfigure}
  \caption{2D stable pendulum: response to sinusoidal input $u(t) = 5 \sin(t/\pi)$ for initial condition $\bx(0) = [1, 1]^\top$ with a degree 7 balanced realization approximation.
  The degree 7 approximation is able to more accurately capture the $\sin(x)$ nonlinearity on this domain than the cubic approximation.}
  \label{fig:example31_sin2_d8}
\end{figure}

The results so far have focused on the local region near the equilibrium point, and already we have seen that truncation errors can render the approximations inherently local.
If we zoom out and look at the transformation of the coordinate grid under the balancing transformation over a larger region of the state space, we can observe another important limitation of the nonlinear approach.
In \cref{fig:example31_zoom}, we compare a linear transformation on the left with a cubic transformation on the right.
Recall that at this point, we have not performed any truncation; therefore, the linear balancing transformation is given by an invertible matrix that makes the mapping globally bijective.
However, once we introduce nonlinearity into the transformation, suddenly the inverse transformation is no longer guaranteed to exist outside of a local neighborhood of the origin.
The curvature introduced by the balancing transformation can result in the grid lines overlapping, in which case the one-to-one property that is critical for a valid coordinate transformation is lost.
As a result, trajectories that stray too far from the origin can start to exhibit anomalous behavior, and eventually numerical instabilities arise.

\begin{figure}[ht!]
  \centering
  \begin{subfigure}[b]{0.44\textwidth}
    \centering
    \includegraphics[width=\linewidth]{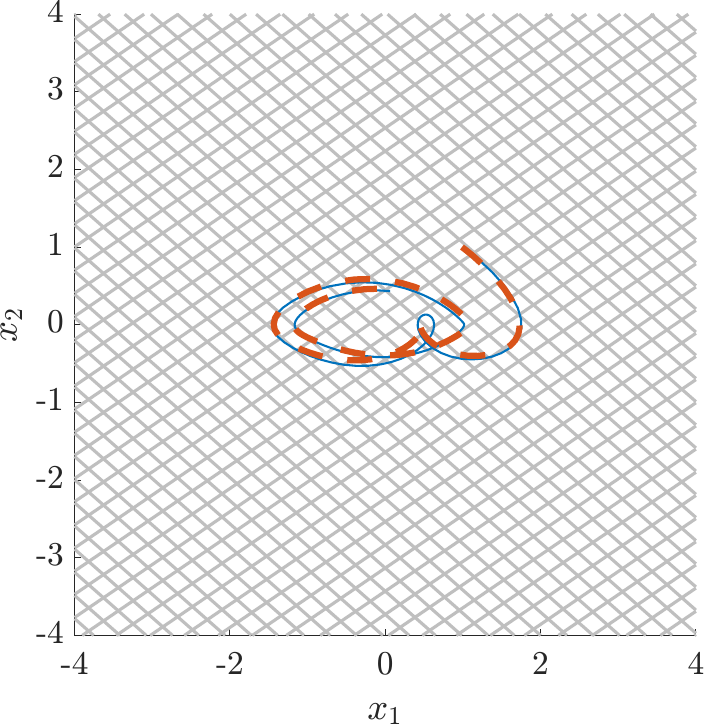}
    \caption{Linear transformation.}
    \label{fig:example31_sin2_x_d2_zoom}
  \end{subfigure}
  \hfill
  \begin{subfigure}[b]{0.44\textwidth}
    \centering
    \includegraphics[width=\linewidth]{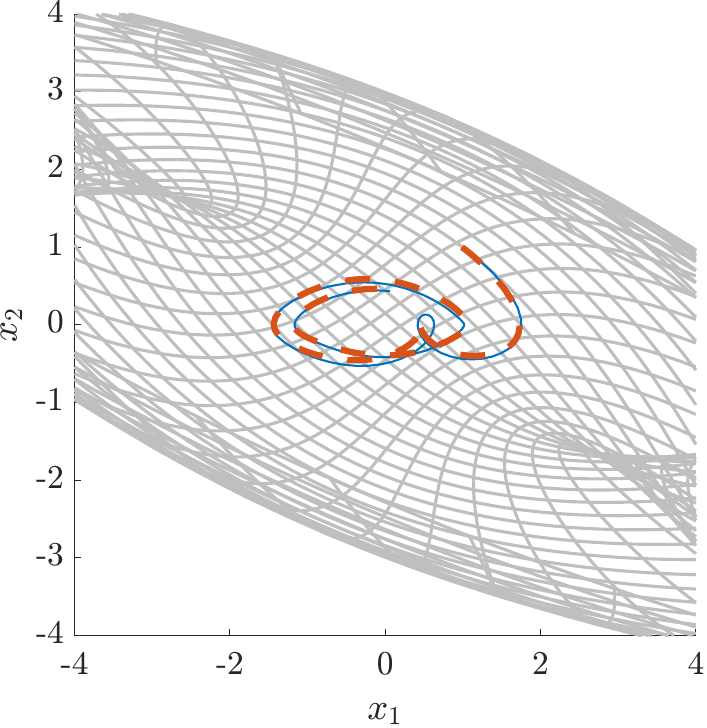}
    \caption{Cubic transformation.}
    \label{fig:example31_sin2_x_d4_zoom}
  \end{subfigure}
  \caption{2D stable pendulum: zoomed out view showing the warping and folding of the coordinate grid under the degree 3 polynomial balancing transformation approximation compared with the globally valid linear transformation. }
  \label{fig:example31_zoom}
\end{figure}

\subsection{3D illustrative example}\label{sec:example3}
In the next example, we illustrate the idea that the nonlinear ROM obtained via nonlinear balanced truncation represents projection of the dynamics onto a curved manifold, as opposed to the linear subspace that is given by linear balanced truncation.
To that end, we consider the following academic model:
\begin{align*}
  \dot{x}_1 & =  - 0.172x_1^3 - 0.172x_1^2 - 0.739x_1 - 0.172x_2^2 + 1.57x_2 - 0.172x_3 + 5.09 u(t),                \\
  \dot{x}_2 & =  1.72x_1^3 + 1.72x_1^2 - 1.57x_1 + 1.72x_2^2 - 6.26x_2 + 1.72x_3 + 4.82 u(t),                       \\
  \dot{x}_3 & = 0.515x_1^2x_2^2 - 1.72x_2 - x_3 - 0.172x_1 + 0.343x_1x_3 - 3.43x_2x_3 + 0.343x_1x_2^2               \\
            & \quad - 8.13x_1^2x_2 + 0.515x_1^2x_3 - 3.43x_1^3x_2 + 0.476x_1^2 + 1.56x_1^3 + 11.5x_2^2 + 0.859x_1^4 \\
            & \quad  - 3.43x_2^3 + 0.515x_1^5 + (0.597 + - 15.3x_1^2 - 10.2x_1 - 9.64x_2 )u(t),                     \\
  y         & = 0.597x_1^3 + 0.597x_1^2 + 5.09x_1 + 0.597x_2^2 - 4.82x_2 + 0.597x_3.
\end{align*}
The computed balancing transformation is
\begin{align*}
  \bx & = \bar\bPhi(\bar \bz) = \begin{bmatrix}
                                  -\bar{z}_1 \\
                                  -\bar{z}_2 \\
                                  \bar{z}_1^3 - \bar{z}_1^2 - \bar{z}_2^2 + \bar{z}_3
                                \end{bmatrix}.
\end{align*}
Truncation of the last state amounts to setting $\bar{z}_3=0$, which defines a manifold given by the set of points
$\cM = \{\bx\in\real^3 : \bx = \bar\bPhi\left(\begin{bmatrix}
      \bar{z}_1, \,
      \bar{z}_2, \,
      0
    \end{bmatrix}^\top\right) \}$ for $\bar{z}_1,\bar{z}_2 \in \real$.
This manifold, along with the linear subspace approximation corresponding to linear balancing, can be seen in \cref{fig:example32_d4}.
Also shown in each of these figures is a trajectory corresponding to the initial condition $\bx(0) = [-1,-2,-4]^\top$ for the FOM (green) and ROM (red).
This initial condition lies on the nonlinear balanced manifold, which indicates that it is among the states that are nonlinearly controllable and observable, and we can see that the nonlinear approach more accurately captures the trajectory.
\Cref{fig:example32_d4_noise} shows a trajectory produced in response to a white noise input, which provides an alternate way to understand which states are reachable for the model \cite{Newman1998}.
We can see that the system's response to white noise lives on the manifold, which confirms that the balanced manifold is capturing the states that are nonlinearly controllable and observable.

\begin{figure}[ht!]
  \centering
  \begin{subfigure}[b]{0.65\textwidth}
    \centering
    \includegraphics[width=.48\linewidth]{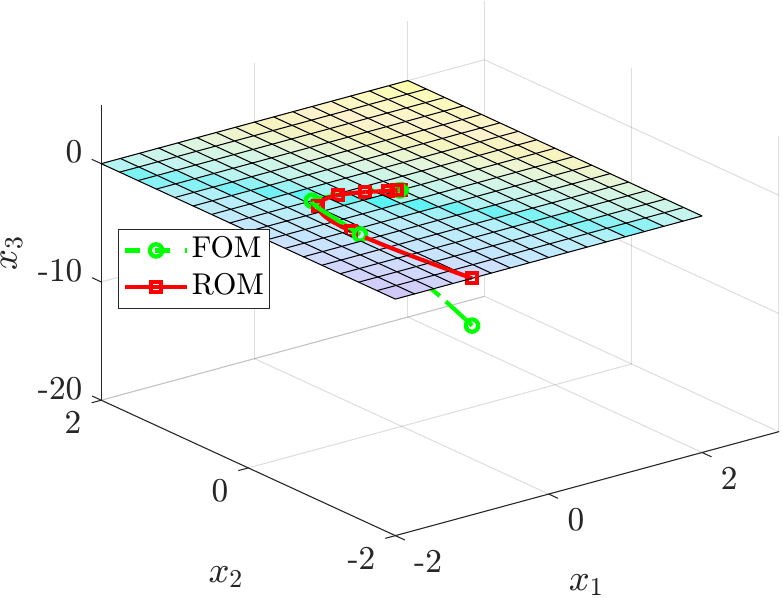}
    \includegraphics[width=.48\linewidth]{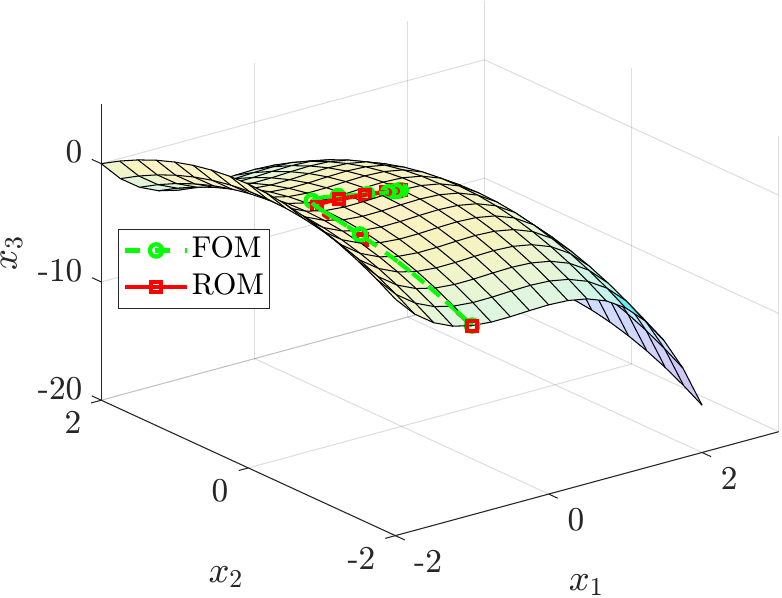}
    \caption{Linear balanced subspace vs. nonlinear balanced manifold, and a simulated unforced response to the initial condition $\bx(0) = [-1,-2,-4]^\top$.}
    \label{fig:example32_d4}
  \end{subfigure}
  \hfill
  \begin{subfigure}[b]{0.325\textwidth}
    \centering
    \includegraphics[width=\linewidth]{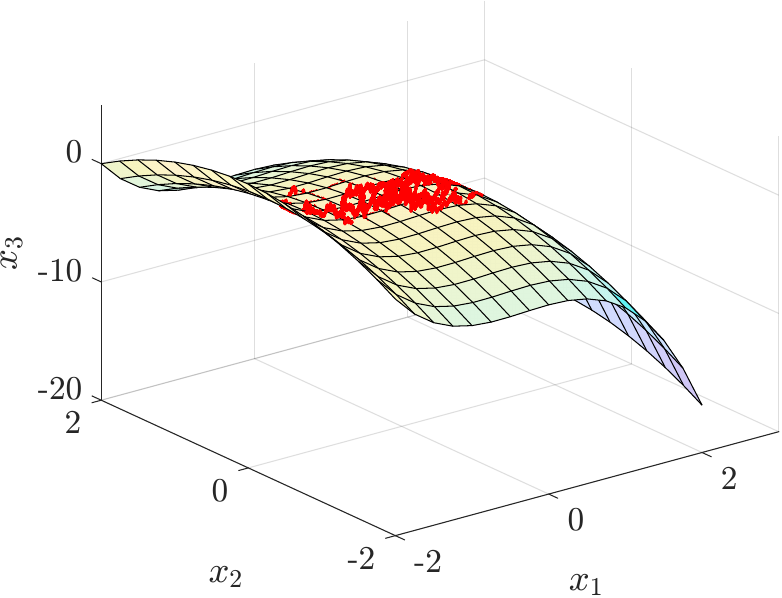}
    \caption{Nonlinear balanced manifold, and a simulated response to white noise excitation.}
    \label{fig:example32_d4_noise}
  \end{subfigure}
  \caption{3D illustrative example: comparison between linear subspace and nonlinear manifold approximations. The curvature of the manifold permits a better reduced-order representation of the controllable and observable states of a nonlinear system. Intuitively, one can think of the controllable states as those that are most readily reached under white noise excitation.}
  \label{fig:example32_manifold}
\end{figure}

In \cref{fig:example32_d4}, it is clear that the state trajectories are more accurately projected on the nonlinear balancing manifold than the linear balancing subspace;
in \cref{fig:example32_d4_y},
we compare the model outputs.
To be clear, we are not comparing
a nonlinear balanced ROM with a \textit{linearized} ROM.
Both ROMs are nonlinear: the difference is the degree of the balancing transformation approximation used, in one case being a linear transformation producing a linear subspace approximation, and in the other being a cubic transformation producing a manifold approximation.
By better capturing the observable states, the nonlinear balancing ROM has lower output error; for this example, the $L_2$ error for the nonlinear balancing ROM is
$0.54$,
whereas the $L_2$ error for the linear balancing ROM is
$2.59$.
The nonlinear manifold projection also more accurately captures the initial condition compared with the linear subspace projection.

\begin{figure}[ht!]
  \centering
  \includegraphics[width=0.85\linewidth]{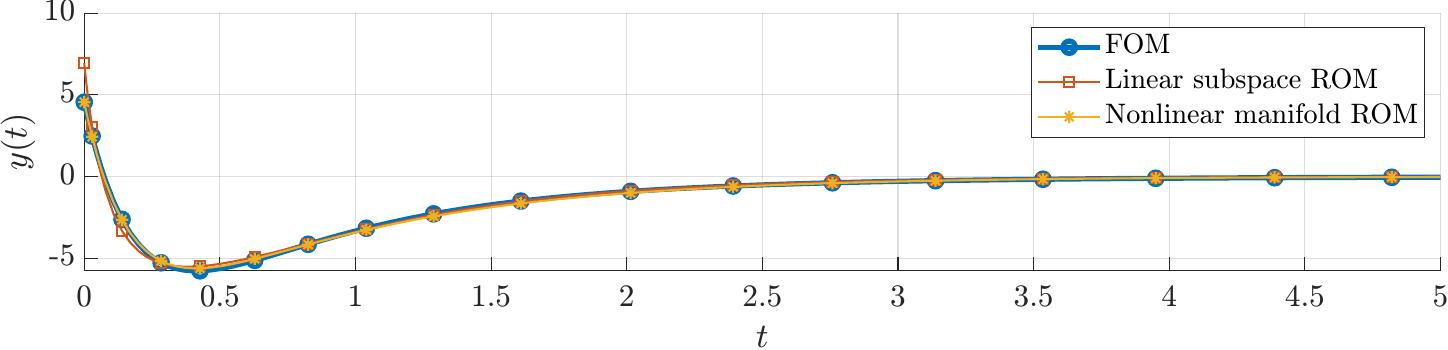}
  \caption{3D illustrative example: comparison of model outputs for linear balanced subspace and nonlinear balanced manifold ROMs.
    The trajectory starts far from the origin and outside of the local balanced subspace due to the curvature of the nonlinear balanced manifold, leading nonlinear balancing to exhibit lower output error. }
  \label{fig:example32_d4_y}
\end{figure}

Finally, we share the explicit balanced realization for the model, along with the ROM.
The balanced realization prior to truncation is
\begin{align*}
  \dot{\bar{z}}_1 & = 1.57\bar{z}_2 - 0.739\bar{z}_1 + 0.172\bar{z}_3  -5.09u(t), \\
  \dot{\bar{z}}_2 & = - 1.57\bar{z}_1 - 6.26\bar{z}_2 - 1.72\bar{z}_3  -4.82u(t), \\
  \dot{\bar{z}}_3 & =  0.172\bar{z}_1 + 1.72\bar{z}_2 - \bar{z}_3 + 0.597u(t),    \\
  y               & = 4.82\bar{z}_2 - 5.09\bar{z}_1 + 0.597\bar{z}_3,
\end{align*}
and the ROM, obtained by setting $\bar{z}_3 = 0$, is
\begin{align*}
  \dot{\bar{z}}_1 & = 1.57\bar{z}_2 - 0.739\bar{z}_1   -5.09u(t), \\
  \dot{\bar{z}}_2 & = - 1.57\bar{z}_1 - 6.26\bar{z}_2  -4.82u(t), \\
  y               & = 4.82\bar{z}_2 - 5.09\bar{z}_1.
\end{align*}
Similar to the 2D illustrative example,
the balanced realization is linear.
In fact, this example was constructed by applying the nonlinear transformation
\begin{align*}
  \bar{\bz} & = \bPsi(\bx)  = \begin{bmatrix}
                                x_1 \\
                                x_2 \\
                                x_3 + x_1^2 + x_2^2 + x_1^3
                              \end{bmatrix},
\end{align*}
to the balanced realization of the following linear time-invariant model from \cite{Holmes2012}:
\begin{align*}
  \dot{x}_1 & = -x_1 + 100 x_3 + u(t),   \\
  \dot{x}_2 & = -2 x_2 + 100 x_3 + u(t), \\
  \dot{x}_3 & = -5 x_3 + u(t),           \\
  y         & = x_1 + x_2 + x_3.
\end{align*}
Therefore, up to a sign change, the computed balancing transformation $\bx = \bar\bPhi(\bar \bz)$ is precisely the inverse of the transformation $\bar{\bz} = \bPsi(\bx)$,
so in this case the balancing transformation is exactly cubic.

\subsection{4D double pendulum}\label{sec:example4}
Our next example is the classic damped double pendulum considered in \cite{Fujimoto2008a}.
As shown in that work, the governing equations can be derived via Lagrange's equations.
The resulting state-space model is given by
\begin{align*}
  \dot{\bx} & = \begin{bmatrix}
                  x_3 \\
                  x_4 \\
                  \bM(\bx)^{-1} \left(
                  \begin{bmatrix}
        \partial L(\bx)/\partial x_1 \\
        \partial L(\bx)/\partial x_2
      \end{bmatrix}
                  - \dot{\bM}(\bx) \begin{bmatrix}
                         x_3 \\ x_4
                       \end{bmatrix}
                  -  \begin{bmatrix}
           \mu_1 x_3 \\ \mu_2 x_4
         \end{bmatrix}
                  \right)
                \end{bmatrix} +
  \begin{bmatrix}
    0 \\ 0 \\ \bM(\bx)^{-1}\begin{bmatrix}
      1 \\ 0
    \end{bmatrix}
  \end{bmatrix} u(t),                                       \\
  \by       & = \begin{bmatrix}
                  \ell_1  \sin(x_1) + \ell_2  \sin(x_1 + x_2) \\
                  \ell_1  \left(1-\cos(x_1)\right) + \ell_2  \left(1-\cos(x_1 + x_2)\right)
                \end{bmatrix},
\end{align*}
where the Lagrangian is given by the kinetic and potential energies as
\begin{align*}
  L(\bx) & = T(\bx) - V(\bx)                                        \\
         & = \frac{1}{2} \begin{bmatrix}
                           x_3 \\ x_4
                         \end{bmatrix}^\top \bM(\bx) \begin{bmatrix}
                                                       x_3 \\ x_4
                                                     \end{bmatrix}
  -
  m_1  g  \ell_1  \cos(x_1)
  - m_2  g (\ell_1  \cos(x_1) + \ell_2  \cos(x_1 + x_2)),
\end{align*}
and the state-dependent symmetric mass matrix is
\begin{align*}
  \bM(\bx) & = \begin{bmatrix}
                 m_1 \ell_1^2 + m_2 \ell_1^2 + m_2 \ell_2^2 + 2 m_2 \ell_1 \ell_2 \cos (x_2) & m_2 \ell_2^2 + m_2 \ell_1 \ell_2 \cos (x_2) \\
                 m_2 \ell_2^2 + m_2 \ell_1 \ell_2 \cos (x_2)                                 & m_2 \ell_2^2
               \end{bmatrix}.
\end{align*}
The input corresponds to a torque at the attachment point of the first link to the ceiling, whereas the outputs represent the horizontal and vertical displacements of the tip of the pendulum.
Additional details can be found in the original model derivations\footnote{Note: there is a typo in \cite{Fujimoto2008a} in which the one-half factor in the kinetic energy is missing.} in \cite{Fujimoto2008a}.
For the model parameters, we take
$g = 9.8$,
$m_1 = m_2 = 1$,
$\ell_1 = \ell_2 = 1$,
and
$\mu_1 = \mu_2 = 1$, as in \cite{Fujimoto2008a}.

As in \cite{Fujimoto2008a}, we compare the outputs
of the FOM with ROMs of order $r=2$
for a sinusoidal input signal $u(t) = \sin(2.5 t)$.
In \cite{Fujimoto2008a}, the authors compare a nonlinear ROM based on nonlinear balancing with a \textit{linearized} ROM based on linear balancing.
While this demonstrates the advantage of nonlinear modeling over linearization, it does not isolate whether the performance gain stems from the nonlinear balancing transformation itself or simply from retaining nonlinear terms in the ROM dynamics.
To decouple these effects,
we add a third ROM to the comparison: a nonlinear ROM obtained by projecting the nonlinear FOM onto the linear balanced subspace.
This permits a more nuanced assessment of the benefits provided specifically by the nonlinear manifold reduction.
\Cref{fig:example15_forcedResponse} shows the outputs of the FOM and the ROMs, where both nonlinear ROMs are of degree~5.

\begin{figure}[htbp]
  \centering
  \includegraphics[width = 0.75\textwidth]{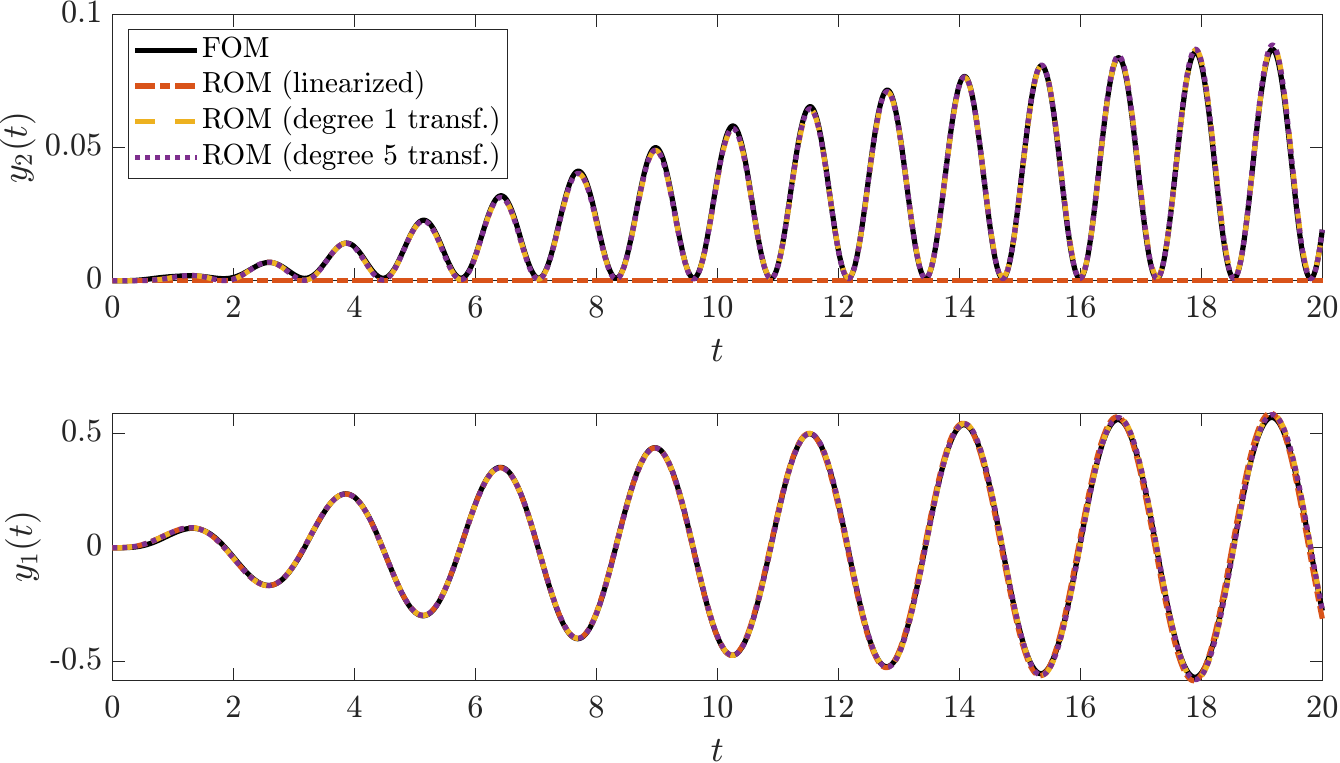}
  \caption{4D double pendulum: comparison of the horizontal ($y_1(t)$) and vertical ($y_2(t)$) tip displacement model outputs for the FOM and ROMs.
    The linearized ROM is the only model that fails to capture the output $y_2(t)$; otherwise, both the linear subspace ROM and the degree 5 balanced manifold ROM work well.}
  \label{fig:example15_forcedResponse}
\end{figure}

As shown in \cref{fig:example15_forcedResponse}, all three ROMs---even the linearized ROM---faithfully reproduce the horizontal tip displacement.
However, the linearized ROM fails entirely to capture the vertical tip displacement dynamics, as also shown in \cite{Fujimoto2008a}.
This is unsurprising, as even the full-order linearized dynamics fail to capture the vertical displacement of the tip, simply because the linearization of $1-\cos(x)$ is $0$.
By including the nonlinear ROM projected onto the linear balancing subspace (omitted in \cite{Fujimoto2008a}),
we see that the output is again well modeled,
showing that it is the linearization of the dynamics rather than the projection onto a linear subspace that hurts the ROM performance.
By visual inspection of the ROM outputs in \cref{fig:example15_forcedResponse}, or by comparing the output errors in \cref{tab:example15},
we see that the nonlinear balanced ROM
only performs marginally better than the linear subspace ROM.
However, this slight improvement in accuracy is accompanied by increased computational complexity, potential for numerical instabilities farther from the origin where the approximations break down, and the additional locality restrictions similar to those illustrated in \cref{fig:example31_zoom}.
In light of these tradeoffs, the linear subspace approximation may be preferable for this example, demonstrating that in some cases,
linear balancing can be an effective approximation of its nonlinear counterpart.
As we will show in \cref{sec:example5}, this is not always the case, and there are examples where nonlinear balancing presents a clear advantage over linear subspace projection.

\begin{table}[htbp]
  \centering
  \caption{4D double pendulum: ROM output errors for horizontal ($y_1(t)$) and vertical ($y_2(t)$) tip displacements. }
  \begin{tabular}{lll}
    \toprule
    ROM Version      & \multicolumn{1}{c}{$y_1$ error} & \multicolumn{1}{c}{$y_2$ error} \\
    \midrule
    Linearized ROM   & $5.33 \cdot 10^{-1}$            & $1.61 \cdot 10^{0}$             \\
    Linear BT ROM    & $2.84 \cdot 10^{-1}$            & $3.07 \cdot 10^{-2}$            \\
    Nonlinear BT ROM & $2.67 \cdot 10^{-1}$            & $2.62 \cdot 10^{-2}$            \\
    \bottomrule
  \end{tabular}
  \label{tab:example15}
\end{table}

\subsection{Nonlinear cantilever beam}\label{sec:example5}
We consider finite element discretization of an Euler--Bernoulli cantilever beam with von Kármán nonlinearity.
The problem setup is shown in \cref{fig:beam}.
\begin{figure}[htbp]
  \centering
  \includegraphics[width = 0.6\textwidth]{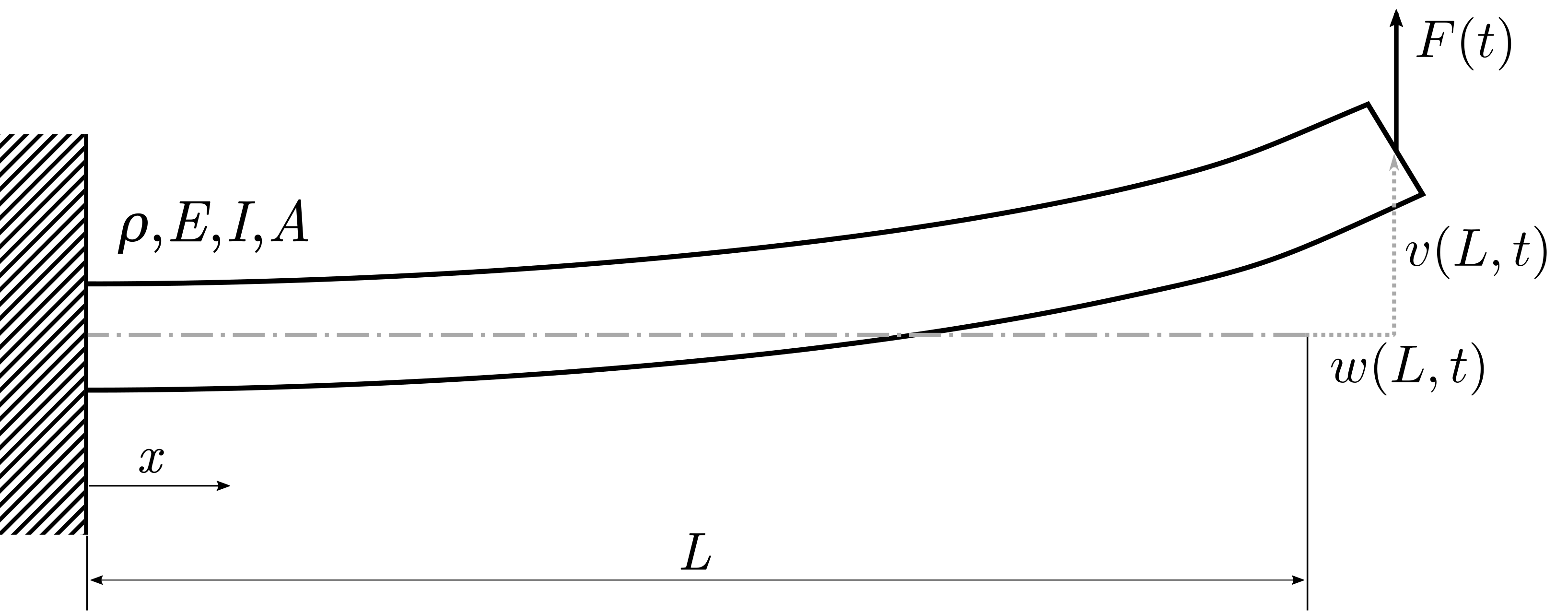}
  \caption{Nonlinear cantilever beam: the deformed beam deviates from the dot-dashed centerline by $w(x,t)$ and $v(x,t)$ in the horizontal and vertical directions, respectively. We are interested in the response of the tip to the transverse forcing $F(t)$.}
  \label{fig:beam}
\end{figure}

\subsubsection{Full-order model derivation}
The governing equations for the nonlinear Euler--Bernoulli beam are
\begin{align*}
  0                                         & = \rho A \frac{\partial^2 v}{\partial t^2} - \frac{\partial }{\partial x} \left(N_{xx} \frac{\partial v}{\partial x} \right)
  + \frac{\partial^2 M_{xx}}{\partial x^2}, &
  0                                         & = \rho A \frac{\partial^2 w}{\partial t^2}
  - \frac{\partial N_{xx}}{\partial x}    ,
\end{align*}
where
$v(x,t)$
and
$w(x,t)$
represent the beam's
transverse
and
longitudinal
deflections, respectively, as functions of position along the beam $x \in [0,L]$ and time $t > 0$.
The quantities $N_{xx}(x,t)$ and $M_{xx}(x,t)$ are the axial force and bending moment.
The von Kármán geometric nonlinearity provides a coupling between transverse deflections and longitudinal strains and is given by the following strain-displacement relations:
\begin{align*}
  N_{xx}(x,t) & = E A \left[\frac{\partial w}{\partial x}  + \frac{1}{2}\left(\frac{\partial v}{\partial x} \right)^2 \right] , &
  M_{xx}(x,t) & = E I \frac{\partial^2 v}{\partial x^2}.
\end{align*}
The model's scalar parameters are the
density $\rho$,
elastic modulus $E$,
second moment of area $I$,
length $L$,
and
cross-sectional area $A$.
The boundary conditions for the fixed end of the beam are
$v(0,t) = w(0,t) = \left[ \partial w(x,t)/\partial x \right]_{x=0} = 0.$
The boundary conditions for the forced end of the beam are
\begin{align*}
  N_{xx}(L,t)                                                                                  & = 0 ,    &
  \left[\frac{\partial v}{\partial x} N_{xx} + \frac{\partial M_{xx}}{\partial x}\right]_{x=L} & = F(t) , &
  M_{xx}(L,t)                                                                                  & = 0.
\end{align*}
Thus, the forcing $F(t)$ will enter though the secondary variables in the finite element formulation \cite{Reddy2004}.

Upon discretization with $N$ finite elements, the model takes the form of a second-order mechanical system
\begin{align*}
  \bM \ddot{\bq} + \bD \dot{\bq} + \bK_1 \bq + \bK_2 \kronF{\bq}{2} + \bK_3 \kronF{\bq}{3} & = \tilde\bB \bu(t),                      &
  \by                                                                                      & = \tilde\bC_1 \bq + \tilde\bC_2 \dot\bq,
\end{align*}
where
$\bM, \bD, \bK_1 \in \real^{3N \times 3N}$,
$\bK_2 \in \real^{3N \times (3N)^2}$,
$\bK_3 \in \real^{3N \times (3N)^3}$,
$\tilde\bB \in \real^{3N \times m}$,
and
$\tilde\bC_1, \tilde\bC_2 \in \real^{p \times 3N}$.
The number of inputs is $m$, and the number of outputs is $p$.
Defining the state vector as the positions and velocities of each node, $\bx = [\bx_1, \bx_2]^\top = [\bq, \dot\bq]^\top$, a state-space realization is then

\begin{align*}
  \dot{\bx} & = \begin{bmatrix}
                  \bx_2 \\
                  -\bM^{-1}\bK_1 \bx_1-\bM^{-1}\bK_2 \kronF{\bx_1}{2}-\bM^{-1}\bK_3 \kronF{\bx_1}{3}-\bM^{-1} \bD \bx_2 + \bM^{-1} \tilde\bB \bu(t)
                \end{bmatrix}, \\
  \by       & = \begin{bmatrix}
                  \tilde\bC_1 & \tilde\bC_2
                \end{bmatrix} \bx.
\end{align*}
This can be rearranged as a cubic polynomial system in Kronecker form as
\begin{align*}
  \dot\bx & = \bA \bx + \bF_2 \kronF{\bx}{2} + \bF_3 \kronF{\bx}{3} + \bB \bu(t), &
  \by     & = \bC \bx.
\end{align*}

The proposed nonlinear balancing implementation is a balance-\textit{then}-reduce procedure, which requires minimality (controllability and observability) of the linearized system.
For example, the computation of the scaling transformation in \cref{sec:polynomial-balancing-transformation} assumes that none of the Hankel singular values are zero in order for the squared singular value functions to be analytic.
The computation of the input-normal/output-diagonal transformation further requires that the Hankel singular values be nonzero and distinct.
This is reminiscent of the linear case, where the original formulation of balanced truncation by Moore \cite{Moore1981} required minimality.
This is a major limitation: reducibility of the model is enabled by the uncontrollable and unobservable modes.
In the linear case, the balance-\textit{and}-truncate procedure of square-root balancing enabled reduction of non-minimal systems \cite{Tombs1987}; we have at present no such formulation for nonlinear balancing.
Therefore, to avoid numerical issues, the FOM must be designed to be both controllable and observable.
We artificially choose $\bB$ and $\bC$ to be identity matrices in order to satisfy this requirement; later, for the purpose of simulating and evaluating the model, we will only apply forcing to the end of the beam and only consider the position of the tip of the beam as output.
The controllability, observability, and stability of the model are also sensitive to the material properties and damping;
we take
$\rho = 8000$ kg/m$^3$,
$E = 210$ GPa,
$I = 0.01$ m$^4$,
$L=1$ m,
and $A = 5$ m$^2$,
and we apply proportional damping given by
$\bD = 0.0001 \bM + 0.0001 \bK_1$.

\subsubsection{One-element model to show the benefit of manifold projection}
First, we consider the case of a single finite element, leading to a state-space model of dimension $n=6$.
The full-order model is
\begin{align*}
  \dot{x}_1 & = u_1(t) +  x_4                                                                                                     ,\qquad
  \dot{x}_2  = u_2(t) +  x_5                                                                                                     ,\qquad
  \dot{x}_3  = u_3(t) +  x_6  ,                                                                                                                               \\
  \dot{x}_4 & = u_4(t)    - 7.88\cdot 10^{7} x_1 - 7880 x_4 - 4.72\cdot 10^{7} x_2^2 + 7.88\cdot 10^{6} x_2 x_3 - 5.25\cdot 10^{6} x_3^2  ,                   \\
  \dot{x}_5 & = u_5(t) +   1.32\cdot 10^{7} x_2 - 1.01\cdot 10^{7} x_3 + 1320 x_5 - 1010 x_6 - 2.05\cdot 10^{8} x_1 x_2                                       \\
            & \quad - 2\cdot 10^{8} x_1 x_3 - 5.91\cdot 10^{7} x_2 x_3^2 - 1.01\cdot 10^{8} x_2^2 x_3 - 1.01\cdot 10^{8} x_2^3 - 5.06\cdot 10^{7} x_3^3 ,     \\
  \dot{x}_6 & = u_6(t) + 1.06\cdot 10^{8} x_2 - 7.75\cdot 10^{7} x_3 + 1.06\cdot 10^{4} x_5 - 7750 x_6 - 8.5\cdot 10^{8} x_1 x_2                              \\
            & \quad  - 1.46\cdot 10^{9} x_1 x_3 - 3.54\cdot 10^{8} x_2 x_3^2 - 9.11\cdot 10^{8} x_2^2 x_3 - 2.02\cdot 10^{8} x_2^3 - 3.57\cdot 10^{8} x_3^3 , \\
  y_1       & = x_1                                                                                                           , \qquad
  y_2       = x_2                                                                                                           , \qquad
  y_3       = x_3                                                                                                           , \qquad
  y_4       = x_4                                                                                                           , \qquad
  y_5       = x_5                                                                                                           , \qquad
  y_6       = x_6.
\end{align*}

The states are, in order: the horizontal, vertical, and rotational positions of the tip, followed by the horizontal, vertical, and rotational velocities of the tip.
The transverse loading $F(t)$ applied to the tip, shown in \cref{fig:beam}, is represented by the control input $u_5(t)$ here.
The primary output of interest $y_1(t) = w(L,t)$ represents the longitudinal displacement of the tip of the cantilever beam, whereas $y_2(t) = v(L,t)$ represents the transverse displacement  of the tip.
Notice that in the linearized dynamics, the states $x_1$ and $x_4$ become decoupled from the other states; this is because longitudinal and transverse deflections are decoupled in linear beam theory.

In \cref{fig:example6_linearized}, we show the natural response of the full-order nonlinear model for an initial condition corresponding to the static deflection shape due to a transverse tip loading.
The figure also shows the linearized response for the same initial condition.
The linearization decays quickly and fails to capture the longitudinal oscillations
resulting from and at exactly twice the frequency of
the transverse oscillations.
This is precisely due to lack of the von Kármán strain coupling in the linearized dynamics.
We seek a ROM that retains this nonlinear effect.
\begin{figure}[htbp]
  \centering
  \includegraphics[width = 0.75\textwidth]{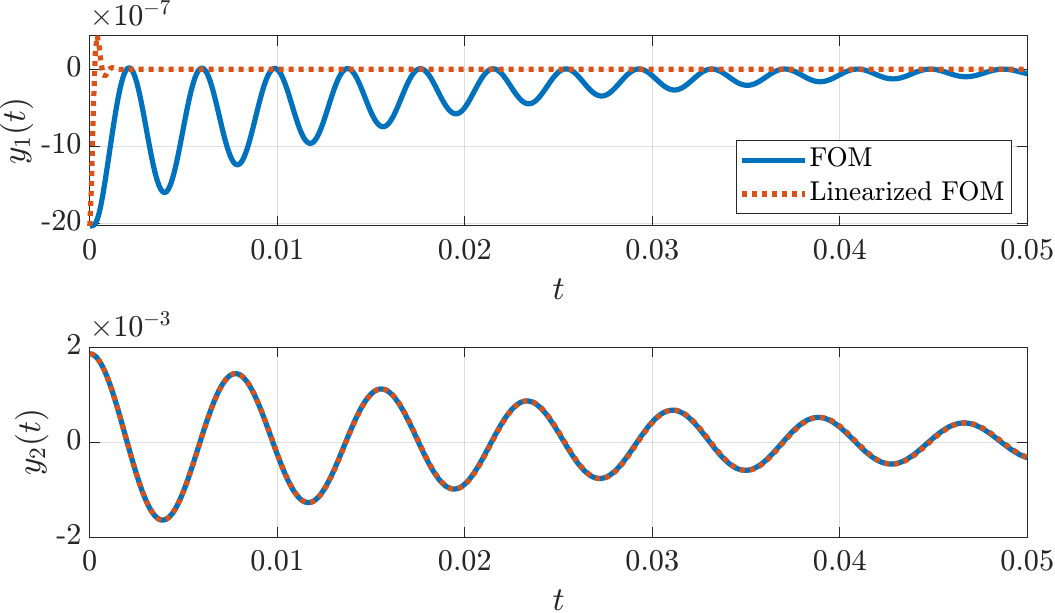}
  \caption{Nonlinear cantilever beam: full-order nonlinear model vs. linearization. Both of these models are formulated in the original coordinate system, but the linearization omits the nonlinear coupling between the longitudinal and transverse dynamics, resulting in the loss of longitudinal displacement information.}
  \label{fig:example6_linearized}
\end{figure}

First, we consider the nonlinear ROM derived by applying linear balancing to the nonlinear model.
The linear balancing transformation is given by the matrix
\begin{align*}
  \bar\bT_1 & = \begin{bmatrix}
                  0     & 0        & 0         & 0                    & -0.011 & 0.00467 \\
                  0.022 & - 0.0202 & - 0.00231 & - 7.96 \cdot 10^{-4} & 0      & 0       \\
                  0.030 & -0.0274  & -0.0164   & -0.00724             & 0      & 0       \\
                  0     & 0        & 0         & 0                    & 63.7   & 63.7    \\
                  -16.6 & -17.7    & 10.3      & -11.9                & 0      & 0       \\
                  -19.5 & -27.9    & 85.4      & -87.7                & 0      & 0
                \end{bmatrix}.
\end{align*}
Since the output map is linear, the composition of the linear output with the linear balancing transformation remains linear, even though the dynamics are nonlinear.
In fact, since the output was chosen to be identity, the output of the balanced realization is given precisely by the balancing transformation.
The first two outputs, representing the longitudinal and transverse tip displacements respectively, are
\begin{equation}\label{eq:linear-output}
  \begin{split}
    y_1 & =    - 0.011 \bar{z}_5 + 0.00467 \bar{z}_6 ,                                               \\
    y_2 & =  0.022 \bar{z}_1 - 0.0202 \bar{z}_2 - 0.00231 \bar{z}_3 - 7.96 \cdot 10^{-4} \bar{z}_4 .
  \end{split}
\end{equation}
In this coordinate system, the longitudinal displacement of the tip of the beam, $y_1$, only depends on the states $\bar{z}_5$ and $\bar{z}_6$;
prior to truncation, these states are still nonlinearly coupled to the other states in the state equation of the dynamics, so the transformed FOM is capable of exhibiting non-zero longitudinal deflections in response to transverse forcing.
However, once we truncate the states $\bar{z}_5$ and $\bar{z}_6$, the output $y_1$ becomes zero and the information is lost, \textit{even though the dynamics remain nonlinear}.
Thus, unlike the pendulum example, the limitation here is in the truncation of a linear balancing transformation that fails to properly encode the redundancy of state components.

If instead we compute a quadratic nonlinear balancing transformation,
the output of the balanced realization is
\begin{align*}
  y_1 & =  - 0.011 \bar{z}_5 + 0.00467 \bar{z}_6 - 2.65 \cdot 10^{-4}  \bar{z}_1^2 + 5.17 \cdot 10^{-4}  \bar{z}_1 \bar{z}_2 - 2.1 \cdot 10^{-4}  \bar{z}_1 \bar{z}_3                   \\
      & \quad  + 0.00269 \bar{z}_1 \bar{z}_4  - 2.6 \cdot 10^{-4}  \bar{z}_2^2 + 1.15 \cdot 10^{-4}  \bar{z}_2 \bar{z}_3 - 0.00213 \bar{z}_2 \bar{z}_4 + 6.2 \cdot 10^{-4}  \bar{z}_3^2 \\
      & \quad - 0.00158 \bar{z}_3 \bar{z}_4 - 0.00342 \bar{z}_4^2,                                                                                                                      \\
  y_2 & =  0.022 \bar{z}_1 - 0.0202 \bar{z}_2 - 0.00231 \bar{z}_3 - 7.96 \cdot 10^{-4}  \bar{z}_4 + 0.00131 \bar{z}_1 \bar{z}_5 - 0.00322 \bar{z}_1 \bar{z}_6                           \\
      & \quad  + 0.00159 \bar{z}_2 \bar{z}_5 - 0.00447 \bar{z}_2 \bar{z}_6 - 0.00151 \bar{z}_3 \bar{z}_5 + 0.00334 \bar{z}_3 \bar{z}_6                                                  \\
      & \quad + 0.00286 \bar{z}_4 \bar{z}_5 + 0.00107 \bar{z}_4 \bar{z}_6,
\end{align*}
where we see that $y_1$ now contains many nonlinear coupling terms.
These additional terms in the balancing transformation are precisely capturing the von Kármán strain relation, encoding the coupling between transverse and longitudinal displacements directly into the output equation.
Therefore, when we truncate the states $\bar{z}_5$ and $\bar{z}_6$, the output $y_1$ retains information.

In \cref{fig:example6_FOMvsROMs}, we compare $T =0.05$ seconds of simulated FOM and ROM responses to initial conditions corresponding to static-deflection shapes of two different tip loads: 0.1 MPa, and 3 MPa.
The smaller initial condition due to smaller loading remains closer to the model's linear regime, whereas the larger initial condition excites the nonlinearities more.
In each of the two cases, the plots on the left compare the original cubic FOM with full-order balanced realization approximations obtained by different degree transformations.
The balanced realizations are all truncated to retain at most cubic dynamics.
On the right, the last two states are truncated from each balanced realization to produce different ROMs.

\begin{figure}[ht!]
  \centering
  \begin{subfigure}[b]{0.49\textwidth}
    \centering
    \includegraphics[width=\linewidth]{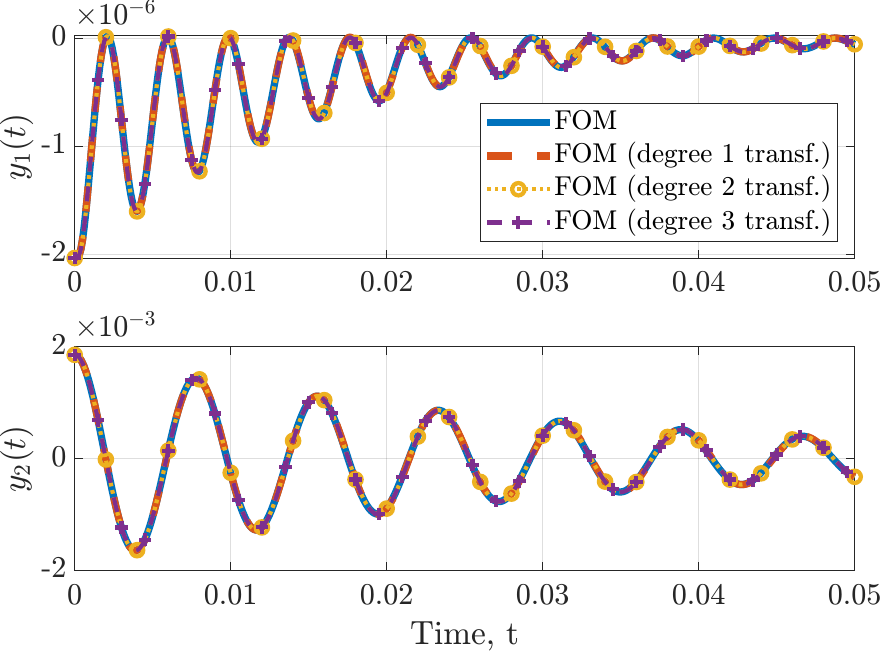}
    \caption{Initial condition 1: FOMs.}
    \label{fig:example6_n6_r6_d4_U0100000_y}
  \end{subfigure}
  \hfill
  \begin{subfigure}[b]{0.49\textwidth}
    \centering
    \includegraphics[width=\linewidth]{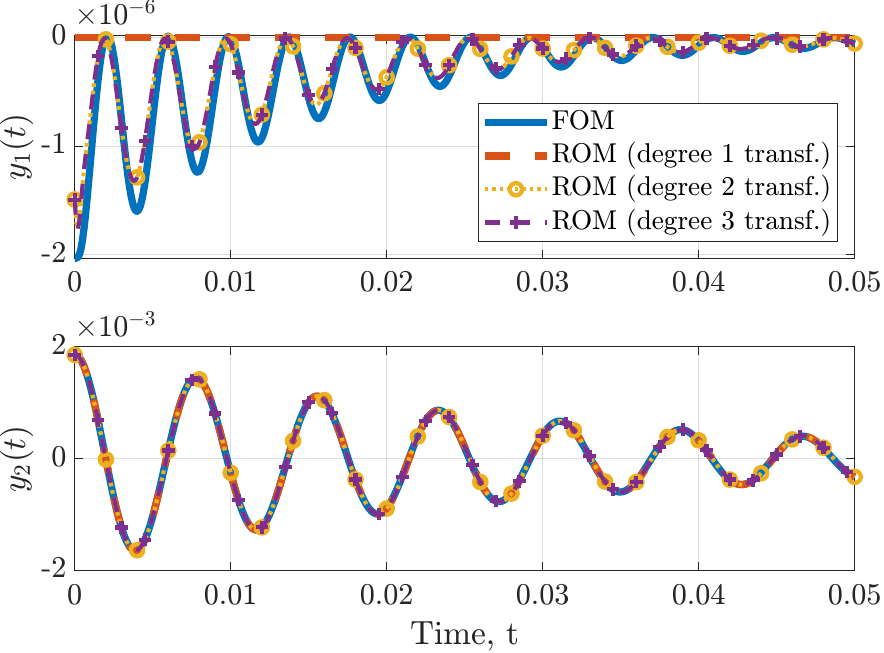}
    \caption{Initial condition 1: FOM vs ROMs.}
    \label{fig:example6_n6_r4_d4_U0100000_y}
  \end{subfigure}

  \begin{subfigure}[b]{0.49\textwidth}
    \centering
    \includegraphics[width=\linewidth]{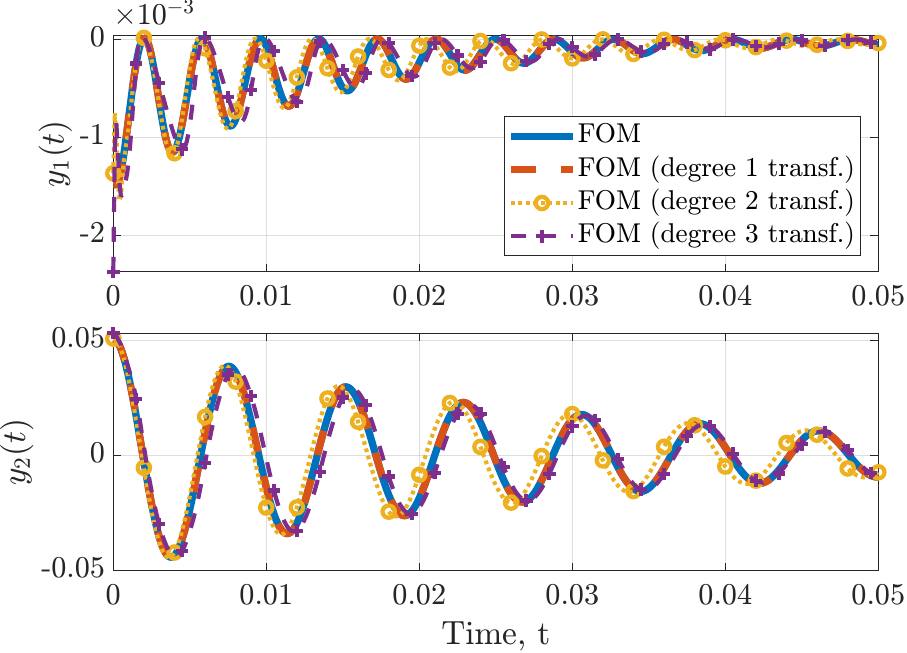}
    \caption{Initial condition 2: FOMs.}
    \label{fig:example6_n6_r6_d4_U03000000_y}
  \end{subfigure}
  \hfill
  \begin{subfigure}[b]{0.49\textwidth}
    \centering
    \includegraphics[width=\linewidth]{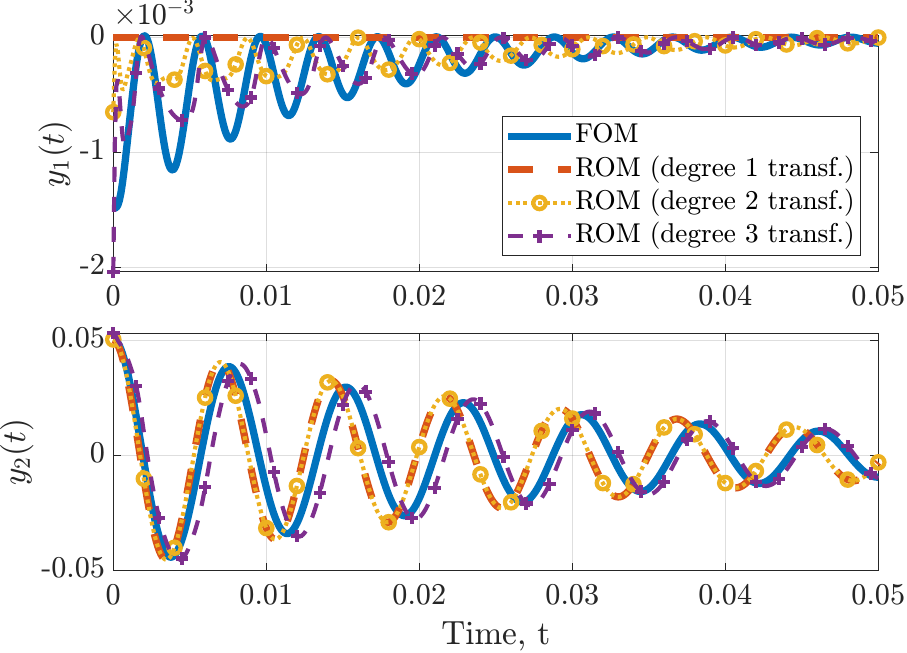}
    \caption{Initial condition 2: FOM vs ROMs.}
    \label{fig:example6_n6_r4_d4_U03000000_y}
  \end{subfigure}
  \caption{Nonlinear cantilever beam: comparison between model outputs for the full-order nonlinear model and various other approximations.
    On the left, the FOM is compared with models obtained by applying linear, quadratic, and cubic balancing transformation approximations \textit{without} truncation.
    On the right, the FOM is compared with the ROMs obtained by truncating the last two states from each of the approximate balanced realizations.
    The output $y_1(t)$ represents the longitudinal displacement of the tip and is more challenging to reproduce, whereas all of the models reproduce the transverse deflection modeled by $y_2(t)$.}
  \label{fig:example6_FOMvsROMs}
\end{figure}

Looking first at \cref{fig:example6_n6_r6_d4_U0100000_y}, we see that for the initial condition closer to the equilibrium point, all three full-order balanced realizations are in agreement.
This is expected, as locally, the nonlinear balancing transformation behaves linearly, and prior to truncation, the balancing transformation should be bijective and should simply represent a coordinate transformation without any introduction of error.
In \cref{fig:example6_n6_r4_d4_U0100000_y}, we see that once states are truncated, the results begin to differ.
As shown previously in \cref{eq:linear-output}, the linear balancing transformation fails to encode the coupling between transverse and longitudinal deflections, so truncating the last two states leads to the output $y_1$ becoming zero.
The nonlinear transformations, which encode some information about the coupling of the dynamics, result in truncated ROMs that still produce an output mirroring the full-order output.

As the initial condition gets further from the origin, we expect to see some differentiation between the truncated linear balancing transformation and the nonlinear approximations of different degrees.
As can be seen in \cref{fig:example6_n6_r6_d4_U03000000_y}, the full-order balanced realizations begin to exhibit some disagreement, particularly in the initial output corresponding to the initial condition.
Unlike linear coordinate transformations, the higher-order polynomial coordinate transformations introduce additional errors even before truncation of states.
First, we always truncate the transformed dynamics to degree 3; thus, for the linear transformation, the dynamics remain exactly degree 3 and do not incur any truncation errors.
However, a cubic transformation of degree 3 dynamics would result in terms up to degree 9, so truncating these higher-order terms can introduce some error further away from the origin.
Second, the polynomial transformations involve more computations which can produce additional rounding errors.
Lastly, and most significantly, for any particular transformation,
translating the initial condition from the original coordinates to the balanced coordinates involves approximating the inverse of the balancing transformation \cref{eq:inverse-balancing-transformation}.
For linear transformations, the exact inverse transformation is readily obtained via the inverse of the linear transformation matrix.
However, the exact inverse of a polynomial transformation is not so readily obtained, nor is it even guaranteed to exist globally.
The approximation we resort to in \cref{thm:poly-inverse} is inherently local, so as we depart from the region of validity of the approximation, we can expect initial condition errors like the ones seen in \cref{fig:example6_n6_r6_d4_U03000000_y}.
When these errors begin to appear, it can be taken as an indication that the balanced realization is near the boundary of its region of validity, and consequently the ROM may not be trustworthy.
In both the FOM simulations in \cref{fig:example6_n6_r6_d4_U03000000_y} and the ROM simulations in \cref{fig:example6_n6_r4_d4_U03000000_y}, we see that the quadratic transformation approximation performs slightly worse than the cubic transformation.

For even larger initial conditions, the stability of the FOM may not even be preserved, as the energy function approximations and balancing transformation approximations all can lie outside of their regions of convergence and lose any rigorous justification.
This represents the major drawback of nonlinear balancing: while improved accuracy may be obtained in a neighborhood of the expansion point at the equilibrium, this comes at the cost of potentially large errors outside of the unknown region of validity of the approximations.

\subsubsection{Scalability with respect to mesh refinement}
In \cref{tab:example6_cpu-scaling,fig:example6_cpu-scaling}, we illustrate the numerical scalability of the proposed approach by scaling up the number of finite elements in the same $T=0.05$ second simulation scenarios shown previously.
In \cite{Corbin2025a,Corbin2025b}, we rigorously proved the computational complexity of the efficient Kronecker-based algorithms
for computing the energy functions and the input-normal/output-diagonal transformation, and
the computations in the final step of obtaining the balanced realization are of the same degree of complexity.
The balanced realization computations do, however, require additional computer memory, so while we were able to compute energy function approximations into the thousands of states in \cite{Corbin2025a}, here we are only able to go into the high hundreds of states, with $n=768$ being the highest for degree~3 energy function approximations and a quadratic balancing transformation approximation, and $n=96$ for degree~4 energy function approximations and a cubic balancing transformation.
Since all of the ROMs computed are nonlinear ROMs with cubic dynamics, one of the most expensive steps is actually applying the balancing transformation to the model operators to obtain the balanced realization expressions.
This step scales with respect to the degree \textit{of the dynamics} in addition to the degree of the transformation, so we see similar scaling regardless of whether a linear, quadratic, or cubic transformation was used to obtain the ROM in \cref{sfig:ROM-CPU-time}.
While the scaling is similar, the overall complexity increases significantly with respect to the polynomial degree.

In \cref{sfig:FOM-ROM-scaling}, we compare the FOM simulation time with each ROM.
The FOM was only simulated up to order $n=48$, as the simulations became too lengthy for larger model sizes.
We see that, as expected, once it comes to online cost, all of the ROMs are significantly faster.
Furthermore, since they are all of order $r=6$ and all feature up to cubic nonlinearities, the simulation times are almost identical, regardless of whether linear, quadratic, or cubic balancing was used in the offline stage.

\pgfplotstableread[col sep=&]{figures/example6_balancingscaling_d3.dat}\tableDataA
\pgfplotstableread[col sep=&]{figures/example6_balancingscaling_d4.dat}\tableDataB
\pgfplotstablecreatecol[copy column from table={\tableDataB}{Balancing CPU-sec}]{Balancing CPU-sec-d4}{\tableDataA}
\pgfplotstablecreatecol[copy column from table={\tableDataB}{ROM Sim. CPU-sec}]{ROM Sim. CPU-sec-d4}{\tableDataA}

\begin{table}[h]
  \centering
  \caption{Nonlinear cantilever beam: scalability results showing the time (in seconds) to simulate the FOM, the time to compute the ROMs, and the time to simulate the ROMs for quadratic and cubic balancing approximations for $T=0.05$ seconds of simulation.}
  \label{tab:example6_cpu-scaling}
  \pgfplotstabletypeset[
  col sep = &,
  every head row/.style={
      before row={
          \toprule
          & & \multicolumn{2}{c}{$d=3$} & \multicolumn{2}{c}{$d=4$} \\ \cmidrule(lr){3-4} \cmidrule(lr){5-6}
        },
      after row=\midrule
    },
  every last row/.style={after row=\bottomrule},
  every column/.style={column type={l}},
  assign column name/.style={/pgfplots/table/column name=\makecell{#1}},
  columns/n/.style={column name=$n$,column type={c}},
  columns/n^4/.style={column name=$n^4$, sci, sci zerofill, precision=2},
  columns/{FOM Sim. CPU-sec}/.style={column name={FOM Sim. \\CPU Time}, sci, sci zerofill, precision=2},
  columns/{Balancing CPU-sec}/.style={column name={Balancing \\CPU Time}, sci, sci zerofill, precision=2},
  columns/{ROM Sim. CPU-sec}/.style={column name={ROM Sim. \\CPU Time}, sci, sci zerofill, precision=2},
  columns/{Balancing CPU-sec-d4}/.style={column name={Balancing \\CPU Time}, sci, sci zerofill, precision=2},
  columns/{ROM Sim. CPU-sec-d4}/.style={column name={ROM Sim. \\CPU Time}, sci, sci zerofill, precision=2},
  columns={{n},{FOM Sim. CPU-sec},{Balancing CPU-sec},{ROM Sim. CPU-sec},{Balancing CPU-sec-d4},{ROM Sim. CPU-sec-d4}},
  ]{\tableDataA}
\end{table}

\pgfdeclareplotmark{emptyhalfcircle}{
  \pgfpathcircle{\pgfpointorigin}{\pgfplotmarksize}
  \pgfusepathqstroke
  \pgfpathmoveto{\pgfpoint{\pgfplotmarksize}{0pt}}
  \pgfpatharc{0}{-180}{\pgfplotmarksize}
  \pgfpathclose
  \pgfusepathqfill
}

\begin{figure}[htb]
  \centering
  \begin{subfigure}[h]{0.485\textwidth}
    \centering
    \begin{tikzpicture}
      \begin{loglogaxis}[
          xlabel={$n$},
          ylabel={CPU sec},ylabel style={yshift=-5pt},
          width=\columnwidth, height=4.5cm,
          xmin=4, xmax=1.2e3,
          ymin=1e-3, ymax=1e4,
          legend style={draw=none,font=\scriptsize\sffamily,legend columns=2,legend to name=legendposcustom1},
          xlabel style={font=\small\sffamily},
          ylabel style={font=\small\sffamily},
        ]
        \addplot[domain=5:2e3, samples=10, color=lightgray, thick, forget plot] {6*10^(-8) * x^3} node [pos=.8, fill=white, sloped] {$n^{3}$};
        \addplot[domain=5:6e2, samples=10, color=lightgray, thick, forget plot] {6*10^(-8) * x^4} node [pos=.8, fill=white, sloped] {$n^{4}$};
        \addplot[domain=5:6e2, samples=10, color=lightgray, thick, forget plot] {10^(-7) * x^5} node [pos=.6, fill=white, sloped] {$n^{5}$};

        \addplot[color=Paired-B,mark=*,mark size=2] table[x={n}, y={Balancing CPU-sec},col sep=&] {figures/example6_balancingScaling_d2.dat};
        \addlegendentry{\makecell{ROM (deg. 1 \\ bal. transf.)}}

        \addplot[color=Paired-A,mark=halfcircle*,mark size=2,mark options={solid, rotate=180}] table[x={n}, y={Balancing CPU-sec},col sep=&] {figures/example6_balancingScaling_d3.dat};
        \addlegendentry{\makecell{ROM (deg. 2 \\ bal. transf.)}}

        \addplot[color=Paired-F,mark=star,mark size=2] table[x={n}, y={Balancing CPU-sec},col sep=&] {figures/example6_balancingScaling_d4.dat};
        \addlegendentry{\makecell{ROM (deg. 3 \\ bal. transf.)}}

      \end{loglogaxis}
      \node[anchor=south] at (current axis.above north) {\ref*{legendposcustom1}};
    \end{tikzpicture}
    \caption{Balanced ROM construction CPU time scaling.}
    \label{sfig:ROM-CPU-time}
  \end{subfigure}
  \hfill
  \begin{subfigure}[h]{0.485\textwidth}
    \centering
    \vspace{9pt}
    \begin{tikzpicture}
      \begin{loglogaxis}[
          xlabel={$n$},
          ylabel={CPU sec},ylabel style={yshift=-5pt},
          width=\columnwidth, height=4.5cm,
          xmin=4, xmax=1.2e3, ymin=1e-3, ymax=1e4,axis on top,
          ytickten={-3,0,3},
          legend style={legend columns=2,legend to name=legendposcustom,
              draw=none,fill=none,font=\sffamily\scriptsize,cells={align=center}},
          xlabel style={font=\small\sffamily},
          ylabel style={font=\small\sffamily},
        ]
        \addplot[domain=5:2e3, samples=10, color=lightgray, thick, forget plot] {6*10^(-8) * x^3} node [pos=.8, fill=white, sloped] {$n^{3}$};
        \addplot[domain=5:6e2, samples=10, color=lightgray, thick, forget plot] {6*10^(-8) * x^4} node [pos=.8, fill=white, sloped] {$n^{4}$};
        \addplot[domain=5:6e2, samples=10, color=lightgray, thick, forget plot] {10^(-7) * x^5} node [pos=.6, fill=white, sloped] {$n^{5}$};

        \addplot[thick,color=black,mark=*,mark size=2] table[x index=1, y index=4,col sep=&] {figures/example6_balancingScaling_d3.dat};
        \addlegendentry{FOM}

        \addplot[color=Paired-B,mark=*,mark size=2] table[x index=1, y index=5,col sep=&] {figures/example6_balancingScaling_d2.dat};
        \addlegendentry{\makecell{ROM (deg. 1 \\ bal. transf.)}}

        \addplot[color=Paired-A,mark=emptyhalfcircle,mark size=2] table[x index=1, y index=5,col sep=&] {figures/example6_balancingScaling_d3.dat};
        \addlegendentry{\makecell{ROM (deg. 2 \\ bal. transf.)}}

        \addplot[color=Paired-F,mark=star,mark size=2] table[x index=1, y index=5,col sep=&] {figures/example6_balancingScaling_d4.dat};
        \addlegendentry{\makecell{ROM (deg. 3 \\ bal. transf.)}}

      \end{loglogaxis}
      \node[anchor=south] at (current axis.above north) {\ref*{legendposcustom}};

    \end{tikzpicture}
    \caption{FOM vs. ROM simulation CPU time scaling. }
    \label{sfig:FOM-ROM-scaling}
  \end{subfigure}

  \caption{Nonlinear cantilever beam: CPU time scaling for computing the balanced ROMs and simulating the FOM and ROMs.
    On the left, we see that the computational scaling of computing the balanced ROMs scales roughly as $O(dn^{d+1})$
    Since all of the ROMs are of order $r=6$, they all scale roughly as the FOMs of order $n=6$. Furthermore, the degree of balancing transformation used does not contribute significantly to the scaling of simulating the model, as the models are all truncated to retain cubic dynamics.
  }
  \label{fig:example6_cpu-scaling}
\end{figure}
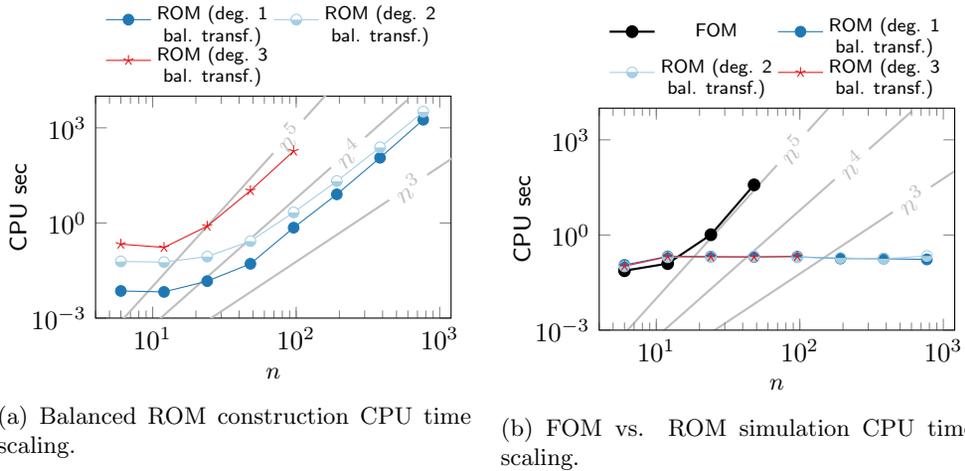

The maximal state dimension considered here of $n=768$ is larger than the largest example considered in \cite{Corbin2025b}, where we stopped at around $n=50$ due to warnings of ill-conditioning.
The same warnings arose in the examples considered herein, but we did not observe a corresponding impact on the ROM performance.
As before, we continued to see ROM performance degrade when solution trajectories departed too far from the local region near the origin, be it through a large initial condition or through strong forcing.
In light of the rest of our results, we present these scalability results not necessarily to recommend applying these methods to practical examples of these state dimensions, but rather to illustrate the full picture that
the requirement of nonzero, distinct Hankel singular values---in the algorithms, but also in the theory itself---is the main obstacle for usability of nonlinear balancing, rather than computational effort.

\section{Conclusions}\label{sec:conclusions}
We have developed new mathematical results and algorithms based on Taylor-series expansions to enable scalable, general purpose implementation of the theory of nonlinear balanced truncation.
Building on our previous works for computing polynomial approximations to the energy functions and the input-normal/output-diagonal transformation \cite{Corbin2025a,Corbin2025b}, in this work, we carry out the additional approximations necessary to fully compute the balancing transformation.
We then derive explicit formulas for the polynomial approximations of the dynamics in the balanced realization.
Upon truncating the transformation, this yields explicit nonlinear reduced-order models.
In our numerical examples, we illustrated the scalability of the approach to systems with moderate state dimensions in the hundreds.
We provide the general purpose \textsc{Matlab} implementation in the \texttt{cnick1/NLbalancing} repository \cite{Corbin2025c}, for the first time enabling the application of nonlinear balancing model order reduction to general problems.

Our results paint a nuanced picture of the benefits and limitations of nonlinear balancing.
We illustrated on low-dimensional examples that while the linear balancing transformation represents a uniform rotation and stretching of the state space, the nonlinear transformation introduces additional curvature and warping of the state space that, in theory, more faithfully captures the dynamics by projecting onto a curved manifold near the expansion point.
Through subsequent examples, we highlighted challenges that arise in accurately capturing this curved manifold.
In particular, we demonstrated that while linear transformations are valid globally, the nonlinear ones are restricted to a neighborhood of the origin, making the balanced realization valid only locally \textit{even prior to truncation}.
This local restriction of the balanced realization makes it difficult to find cases where the nonlinear theory demonstrates marked improvement over the linear theory, which itself is by definition a good approximation of the nonlinear theory locally.
The nonlinear beam represented the best result we were able to construct, in which the linearly balanced ROM failed to capture nonlinear dynamic coupling in the states, leading to a complete loss of information on some outputs of interest.
Conversely, many of the other examples we considered, such as the double pendulum, showed nearly indistinguishable results when comparing linearly vs nonlinearly balanced ROMs, making it difficult to justify the additional complexity of the nonlinear method.

The algorithms proposed in this work represent a balance-\textit{then}-truncate approach that, similar to the case in linear balancing, is ill-conditioned in the case of a minimal or nearly minimal realization featuring zero or near-zero Hankel singular values.
This presents a major limitation, as zero and near-zero singular values are what indicate reducibility.
In our testing, we found that model parameters and, in particular, control and measurement locations had to be artificially chosen in order to ensure that singular values were nonzero and distinct in order to avoid numerical issues.
The development of a balance-\textit{and}-reduce approach akin to square-root balancing in the linear case could enable nonlinear balancing of non-minimal systems in a way that alleviates this numerical ill-conditioning.
Similar to the linear case, this could also open the door to a low-rank adaptation to enable scalability to very high dimensions.

\appendix

\section{Implicit balanced realization via Newton iterations}\label{sec:Newton-balancing}
In this section, we describe an alternative approach for evaluating the nonlinear balanced realization described in \cref{sec:balanced-realization}.
We take the same starting point as \cref{sec:polynomial-balancing} of the main text:
using the methods described in
\cite{Corbin2025b},
we assume access to polynomial approximations for the input-normal/output-diagonal transformation $\bx = \bPhi(\bz)$ and the squared singular value functions $\bsigma^2(\bz)$.
It is possible to evaluate the balanced realization with just these quantities and without applying additional polynomial approximations,
although it requires the use of Newton iterations.
To simulate the balanced realization \cref{eq:FOM-balanced}, we need to evaluate:
\begin{enumerate}
  \item The inverse transformation $\bar\bz_0 = \bar\bPhi^{-1}(\bx_0)$,
        e.g. to find the initial condition $\bar\bz_0$ in the transformed coordinates corresponding to $\bx_0$.
  \item The forward transformation $\bx = \bar\bPhi(\bar\bz)$,
        e.g. to find the $\bx$ corresponding to $\bar\bz$ at each timestep to evaluate the dynamics using $\bf(\bx)$, $\bg(\bx)$, and $\bh(\bx)$.
  \item The Jacobian $\frac{\partial \bar\bPhi(\bar\bz)}{\partial \bar\bz}$, e.g. to evaluate
        $ \dot{\bar\bz} = \left[\frac{\partial \bar\bPhi(\bar\bz)}{\partial \bar\bz}\right]^{-1} \dot{\bx}$.
\end{enumerate}
In the following subsections, we will break down how these tasks are carried out using Newton iterations

\subsection{Evaluating the forward balancing transformation}
The balancing transformation is the composition of two transformations, so we can write it as $\bx = \bar\bPhi(\bar\bz) = \bPhi(\bvarphi(\bar\bz))$, in which case it becomes clear that we must evaluate this in two steps: a) evaluate the scaling transformation $\bz = \bvarphi(\bar\bz)$ (for a given $\bar\bz$), and then b) evaluate the input-normal/output-diagonal transformation $\bx = \bPhi(\bz)$.
The first task is described in \cref{sec:scaling-transformation}, and the second is described in \cref{sec:input-normal-transformation}.

\subsubsection{Evaluating the scaling transformation}\label{sec:scaling-transformation}
The scaling transformation and its inverse, as described in \cref{cor:scaling-1}, are
\begin{align*}
  \bar\bz & = \bvarphi^{-1}(\bz) = \bz \odot \sqrt{\bsigma(\bz)} ,                &
  \bz     & = \bvarphi(\bar\bz) = \bar\bz \odot \sqrt{\bar\bsigma(\bar\bz)}^{-1}.
\end{align*}
Since we only have access to the squared singular value functions in the input-normal/output-diagonal coordinates $\bsigma^2(\bz)$, we can evaluate the inverse scaling transformation $\bvarphi^{-1}(\bz)$ more readily than the scaling transformation $\bvarphi(\bz)$ itself.
The inverse scaling transformation can be evaluated for a fixed $\bz$ straightforwardly by rewriting it as
$\bar\bz = \bvarphi^{-1}(\bz) = \bz \odot ({\bsigma^2(\bz)})^{1/4}$.

Therefore, we must solve the
following nonlinear equation:
\begin{align*}
  f(\bz) \coloneqq \bz \odot \sqrt{\bsigma(\bz)} - \bar\bz = \bzero.
\end{align*}
Solving $f(\bz) = \bzero$ via Newton iteration is done by iterating the equation
\begin{align}\label{eq:Newton-iteration}
  \bz_{n+1} & = \bz_{n} - \bJ(\bz_{n})^{-1} f(\bz_{n}),
\end{align}
so in addition to $f(\bz)$ we need to evaluate the Jacobian $\bJ(\bz) = \frac{\partial f(\bz)}{\partial \bz} $.
The Jacobian for $f(\bz) \coloneqq \bz \odot \sqrt{\bsigma(\bz)} - \bar\bz$ is the same as the Jacobian for $\bvarphi^{-1}(\bz)$, since $\bar\bz$ is fixed.
Knowing $\bvarphi^{-1}(\bz)$, we can compute its Jacobian at $\bz=\bz_n$ by the product rule as
\begin{align}
  \bJ(\bz_n) & = \left.\frac{\partial \bvarphi^{-1}(\bz)}{\partial \bz} \right\vert_{\bz=\bz_n}  = \text{diag} \left(\sqrt{\bsigma(\bz_n)} + \frac{1}{2}\bz_n \odot \sqrt{\bsigma(\bz_n)}^{-1} \odot \bsigma'(\bz_n) \right). \label{eq:jacobian-inverse-scaling}
\end{align}
Since the scaling transformation is decoupled, i.e. the $i$th component of the transformation only depends on the $i$th state, this Jacobian matrix is a diagonal matrix.
All of these quantities are known:
$\bz_n$ is the current iterate, the first choice being some initial guess such as $\bz_1 = \bzero$.
Given $\bz_n$, we can compute $\sqrt{\bsigma(\bz_n)}$ as the fourth root of our polynomial approximation to the squared singular value function $\bsigma^2(\bz_n)$, and $\sqrt{\bsigma(\bz_n)}^{-1}$ is $1/\sqrt{\bsigma(\bz_n)}$ evaluated entry-wise.
Finally, using the available polynomial expansion for $\bsigma^2(\bz_n)$, its derivative  $\frac{\rd}{\rd z} \bsigma^2(\bz_n)$ can be used to evaluate
$\bsigma'(\bz_n)$ as
$\frac{\rd}{\rd z} \bsigma(\bz_n) = \frac{\rd}{\rd z} \sqrt{\bsigma^2(\bz_n)} = \frac{1}{2\bsigma(\bz_n)}\frac{\rd}{\rd z} \bsigma^2(\bz_n)$.
With the Jacobian, the current iterate $\bz_n$ can be used to compute the next iterate $\bz_{n+1}$ via the Newton iteration described by \cref{eq:Newton-iteration}.
Subject to stopping and convergence criteria, this provides a means for evaluating the scaling transformation.

\subsubsection{Evaluating the input-normal/output-diagonal transformation}\label{sec:input-normal-transformation}
After computing $\bz = \bvarphi(\bar\bz)$, evaluating $\bx = \bPhi(\bz)$
can be done efficiently Kronecker product identities,
since the input-normal/output-diagonal transformation coefficients are known explicitly.
Combining the scaling transformation and the input-normal/output-diagonal transformation, the full balancing transformation can be evaluated at each time step to aid in evaluating the balanced realization.

\subsection{Evaluating the inverse balancing transformation}

The approach for evaluating the forward balancing transformation described in the previous section can be used to aid in evaluating the inverse balancing transformation, namely to be able to transform the original initial condition to the balanced coordinates.
Similar to the scaling transformation, we can invert the balancing transformation by solving the following nonlinear equation:
\begin{align*}
  f(\bar\bz) \coloneqq \bar\bPhi(\bar\bz) - \bx = \bzero.
\end{align*}
The solution to this equation is obtained using Newton iteration:
\begin{align}
  \bar\bz_{n+1} & = \bar\bz_{n} - \bJ(\bar\bz_{n})^{-1} f(\bar\bz_{n}). \label{eq:Newton-iteration-2}
\end{align}
The Jacobian for $f(\bar\bz) \coloneqq \bar\bPhi(\bar\bz) - \bx$ is the same as the Jacobian for the balancing transformation $\bar\bPhi(\bar\bz)$, since $\bx$ is fixed.
Hence, the last challenge is that of computing the Jacobian of the balancing transformation, described next in \cref{sec:balancing-jacobian}.

\subsection{Evaluating the Jacobian of the balancing transformation}\label{sec:balancing-jacobian}
The Jacobian of the balancing transformation can be computed via the chain rule as
\begin{align*}
  \bJ(\bar\bz_{n}) & = \frac{\partial \bar\bPhi(\bar\bz)}{\partial \bar\bz} = \frac{\partial \bPhi(\bvarphi(\bar\bz))}{\partial \bar\bz}                                                      = \left[\frac{\partial \bPhi(\bz)}{\partial \bz} \right]_{\bz = \bvarphi(\bar\bz)} \left[\frac{\partial \bvarphi(\bar\bz)}{\partial \bar\bz} \right]_{\bar\bz=\bar\bz}.
\end{align*}
This can be approached in two steps: a) compute the matrix $\left[\frac{\partial \bvarphi(\bar\bz)}{\partial \bar\bz} \right]_{\bar\bz=\bar\bz}$, and then b) compute the matrix $\left[\frac{\partial \bPhi(\bz)}{\partial \bz} \right]_{\bz = \bvarphi(\bar\bz)}$.
The balancing transformation Jacobian is then the product of the two matrices.

\subsubsection{Evaluating the Jacobian of the scaling transformation}
\Cref{eq:jacobian-inverse-scaling} describes how to compute the
Jacobian of the inverse scaling transformation $\left[\frac{\partial \bvarphi^{-1}(\bz)}{\partial \bz} \right]_{\bz = \bvarphi(\bar\bz)}$,
and we can compute $\bz$ given $\bar\bz$ from \cref{sec:scaling-transformation}.
Since the Jacobian of the inverse scaling transformation is the inverse of the Jacobian of the scaling transformation,
the Jacobian of the scaling transformation can be computed as the inverse of the Jacobian of the inverse scaling transformation.
These Jacobians are diagonal matrices, so they are relatively easy to invert; nonetheless, this process still involves Newton iterations.

\subsubsection{Evaluating the Jacobian of the input-normal/output-diagonal transformation}
Once we have $\bz = \bvarphi(\bar\bz)$, evaluating $\left[\frac{\partial \bPhi(\bz)}{\partial \bz} \right]_{\bz = \bvarphi(\bar\bz)}$ is straightforward since the polynomial transformation coefficients are known.
The Jacobian of the input-normal/output-diagonal transformation is given by
\begin{align}
  \left[\frac{\partial \bPhi(\bz)}{\partial \bz} \right]_{\bz = \bvarphi(\bar\bz)} = \bT_1 + 2\bT_2 \left(\bI \otimes \bz \right) + \dots \label{eq:input-normal-jacobian}
\end{align}
which can be evaluated somewhat efficiently using Kronecker product identities.

\subsection{Model order reduction by truncation}
The full-order balanced realization
\cref{eq:FOM-balanced}
is given by the model operators
\cref{eq:balanced-realization-operators},
which are evaluated implicitly via Newton iterations as described in the previous section.
At each time step in the simulation, the balancing transformation $\bx = \bar\bPhi(\bar\bz)$ is evaluated so that the drift and input maps $\bf(\bx)$ and $\bg(\bx)$ can be evaluated, and then the inverse of the Jacobian is evaluated using Newton iterations and used to pre-multiply the state equation to obtain the time derivative of the balanced state $\dot{\bar\bz}$.
The reduced-order balanced realization is given by
\begin{align*}
  \dot{\bar\bz}^a & = \bI_{r\times n}
  \left[\begin{array}{@{}c@{}}\displaystyle\frac{\partial \bar\bPhi\left(\begin{bmatrix} \bar\bz^a \\ \bzero \end{bmatrix}\right)}{\partial \bar\bz} \end{array} \right]
  ^{-1} \bf(\bar\bPhi\left(\begin{bmatrix}
                                 \bar\bz^a \\
                                 \bzero
                               \end{bmatrix}\right)) + \bI_{r\times n}
  \left[\begin{array}{@{}c@{}}\displaystyle\frac{\partial \bar\bPhi\left(\begin{bmatrix} \bar\bz^a \\ \bzero \end{bmatrix}\right)}{\partial \bar\bz} \end{array} \right]
  ^{-1} \bg(\bar\bPhi\left(\begin{bmatrix}
                                 \bar\bz^a \\
                                 \bzero
                               \end{bmatrix}\right)) \bu,           \\
  \by             & = \bh(\bar\bPhi\left(\begin{bmatrix}
                                               \bar\bz^a \\
                                               \bzero
                                             \end{bmatrix}\right)),
\end{align*}
which still retains the implicit structure, since explicit expressions for the balancing transformation and its Jacobian were not derived in this approach.
Among the many inefficiencies of the Newton iteration approach, the lifting-type bottleneck is the most severe, as
evaluation of the dynamics involves the transformation of the reduced state to the full-order state dimension, evaluation of the full-order drift and input maps,
evaluation and inversion of the full-order state-dependent Jacobian,
and then projection back down to the reduced-order dimension.
Even for linear dynamics, this approach does not produce an explicit representation of the reduced-order model.
\begin{algorithm}[h!]
  \caption{Evaluation of the balanced realization via Newton iteration}\label{alg:alg3}
  \begin{algorithmic}[1]
    \Require
    FOM dynamics $\bf(\bx)$, $\bg(\bx)$, $\bh(\bx)$ in \cref{eq:FOM-NL};
    balanced realization state $\bar\bz$ at a particular time step.
    \Ensure
    Balanced realization state derivative $\dot{\bar\bz}$;
    output $\by$.
    \State Evaluate the forward balancing transformation $\bx = \bar\bPhi(\bar\bz)$ to get $\bx$:
    \begin{algsubstates}
      \State Compute $\bz$ given $\bar\bz$ by iterating \cref{eq:Newton-iteration}.
      \State Compute $\bx$ given $\bz$ by evaluating \cref{eq:input-normal}.
    \end{algsubstates}
    \State Compute the Jacobian of the balancing transformation at $\bx$
    \begin{algsubstates}
      \State Compute the Jacobian of the scaling transformation. This can be computed by inverting the Jacobian of the inverse of the scaling transformation \cref{eq:jacobian-inverse-scaling}.
      \State Compute the Jacobian of the input-normal/output-diagonal transformation \cref{eq:input-normal-jacobian}
      \State Multiply the two Jacobians.
    \end{algsubstates}
    \State Evaluate the dynamics using $\bf(\bx)$ and $\bg(\bx)$, then multiply by the inverse of the Jacobian of the balancing transformation to get $\dot{\bar\bz}$.
    \State Evaluate $\bh(\bx)$ in the original coordinates to get the output.
  \end{algorithmic}
\end{algorithm}

\section{Accuracy of the polynomial approach vs. the Newton iteration approach}\label{sec:polynomial-vs-Newton}
On its surface, the polynomial approach in \cref{sec:polynomial-balancing} appears to introduce additional layers of approximation, so it may appear that the Newton iteration approach is more accurate.
However, in this section, we aim to make the argument that both methods exhibit the same degree of accuracy due to the truncation of higher-order terms.

\begin{proposition}
  The polynomial approach proposed in \cref{sec:polynomial-balancing} and the Newton iteration approach proposed in appendix \ref{sec:Newton-balancing} both have errors of $O(\kronF{\bx}{d})$, where $d$ is the degree chosen for the energy function approximation.
\end{proposition}
\begin{proof}
  Assume a control-affine polynomial system is given of degree $d-1$:
  \begin{align*}
    \dot{\bx} & = \bA\bx + \dots + \bF_{d-1}\kronF{\bx}{d-1} + \left( \bB + \dots + \bG_{d-2} \left(\bI_m \otimes \kronF{\bx}{d-2}\right) \right) \bu(t), \\
    \by       & =  \bC\bx + \dots + \bH_{d-1}\kronF{\bx}{d-1}.
  \end{align*}
  At the first step of balancing, we compute polynomial approximations to the energy functions, which we generally must truncate to a degree $d$:
  \begin{align*}
    \cE_c(\bx) \approx \frac{1}{2} \sum_{i=2}^d \bv_i^\top \kronF{\bx}{i}, \qquad
    \cE_o(\bx) \approx \frac{1}{2} \sum_{i=2}^d \bw_i^\top \kronF{\bx}{i}.
  \end{align*}
  Higher-degree terms in the dynamics do not affect the approximations of the energy functions, which have error of $O(\kronF{\bx}{d+1})$.
  Critically, if higher-degree terms exist in the dynamics but are not included, any energy function approximation of degree higher than $d$ will still retain errors of $O(\kronF{\bx}{d+1})$,
  so continuing to compute higher-order approximations without all of the necessary information does not improve the accuracy from a Taylor approximation perspective.
  Continuing with the approximations, the input-normal/output-diagonal and balancing transformations can be computed to degree $d-1$ as
  \begin{align*}
    \bar\bPhi(\bar\bz) & \approx \bT_1 \bar\bz + \bT_2 \kronF{\bar\bz}{2} + \dots  +  \bT_{d-1} \kronF{\bar\bz}{d-1},
  \end{align*}
  where, again, computing any higher-order terms will not eliminate the $O(\kronF{\bx}{d})$
  error if the corresponding terms in the dynamics were truncated.
  Thus, when we finally reach the balanced realization, we can compute terms up to degree $d-1$ to obtain
  \begin{align*}
    \dot{\bar{\bz}} & = \bar\bA\bar{\bz} + \dots + \bar\bF_{d-1}\kronF{\bar{\bz}}{d-1} + \left( \bar\bB + \dots + \bar\bG_{d-2} \left(\bI_m \otimes \kronF{\bar{\bz}}{d-2}\right) \right) \bu(t), \\
    \by             & =  \bar\bC\bar{\bz} + \dots + \bar\bH_{d-1}\kronF{\bar{\bz}}{d-1}.
  \end{align*}
  and, again, computing any higher-order terms will not eliminate the $O(\kronF{\bx}{d})$ error.

  Now, assume the dynamics are of a fixed polynomial degree $k$ strictly less than $d-1$; in this case, we can continue computing higher-order terms to the energy functions to continue getting a more accurate answer.
  However, once the degree $d$ of the energy functions is fixed, the maximum level of accuracy for the transformations and for the balanced realization is fixed, as a more accurate approximation will always require access to the degree $d+1$ term in the energy function expansions.
  Thus, from the polynomial perspective, the approximation we compute is as good as possible based on the preceding truncated polynomial approximations.

  Therefore, the approximation computed using the polynomial realization approach in the main body of the text and the implicit realization computed using Newton iterations both have errors of $O(\kronF{\bx}{d})$, where $d$ is the degree chosen for the energy function approximation.
  It is possible that different approximations may be better behaved in terms of numerical stability, etc., but the local accuracy of both methods is of the same order.
\end{proof}

\bibliographystyle{abbrv}
\bibliography{references}
\end{document}